\documentclass{scrartcl}

\usepackage{amsmath,amssymb}
\usepackage{amsthm}
\usepackage{graphicx}
\usepackage{subfigure}

\usepackage[utf8]{inputenc}
\usepackage[T1]{fontenc}
\usepackage{lmodern}

\usepackage{tikz}
\usetikzlibrary{decorations.markings}
\usepackage[arrow, matrix, curve]{xy}

\usepackage{hyperref}

\usepackage{slashed} 

\usepackage[toc,page]{appendix}
\usepackage{placeins}

\usepackage{mathrsfs} 
\usepackage{cases} 
\usepackage[nocompress]{cite} 

\numberwithin{equation}{section}

\newtheorem{theorem}{Theorem}[section]
\newtheorem{corollary}[theorem]{Corollary}
\newtheorem{lemma}[theorem]{Lemma}
\newtheorem{proposition}[theorem]{Proposition}

\theoremstyle{definition}

\newtheorem{remark}[theorem]{Remark}

\newtheorem*{theorem*}{Theorem}

\begin{document}
\title{Short time existence of the heat flow for Dirac-harmonic maps on closed manifolds}
\author{Johannes Wittmann}
\maketitle
\begin{abstract}
	The heat flow for Dirac-harmonic maps on Riemannian spin manifolds is a modification of the classical heat flow for harmonic maps by coupling it to a spinor. For source manifolds with boundary it was introduced in \cite{ChenJostSunZhu} as a tool to get a general existence program for Dirac-harmonic maps, where also short time existence was obtained. The existence of a global weak solution was established in \cite{globalweaksol}. We prove short time existence of the heat flow for Dirac-harmonic maps on closed manifolds.
\end{abstract}
\tableofcontents
\newpage

\section{Introduction}
\subsection{Dirac-harmonic maps}
Dirac-harmonic maps, introduced in \cite{Dhmaps}, are the critical points of a functional motivated by the supersymmetric non-linear sigma model from quantum field theory.

More precisely, let $M$ be a compact Riemannian spin manifold with fixed spin structure and $N$ a compact Riemannian manifold. We denote by $\Sigma M$ the complex spinor bundle of $M$. For maps $f\colon M\to N$ and spinors $\psi\in\Gamma(\Sigma M\otimes_{\mathbb{R}}f^*TN)$ we consider the functional
\[(f,\psi)\mapsto\frac{1}{2}\int_M\Big(\|df\|^2+(\psi, \slashed{D}^f\psi)\Big)\,dV  \] 
where $(.,.)$ is the inner product induced by the real part of the natural hermitian inner product on $\Sigma M$ and the Riemannian metric on $N$. Moreover, $\slashed{D}^f$ is the Dirac operator of the twisted Dirac bundle $\Sigma M\otimes f^*TN$. Locally,
\[\slashed{D}^f\psi= (\slashed{D}\psi^i)\otimes s_i + (e_\alpha\cdot \psi^i)\otimes\nabla^{f^*TN}_{e_\alpha}s_i\]
where $\psi=\psi^i\otimes s_i$, the $\psi^i$ are local sections of $\Sigma M$, $(s_i)$ is a local frame of $f^*TN$, $(e_\alpha)$ is a local orthonormal frame of $TM$, $\nabla^{f^*TN}$ is the pull-back of the Levi-Civita connection on $TN$, and $\slashed{D}$ is the usual Dirac operator acting on sections of $\Sigma M$. We say that $\slashed{D}^f$ is the \textit{Dirac operator along the map $f$}. 

The critical points of the above functional are called \textit{Dirac-harmonic maps}. They are characterized by the equations
\begin{align}\label{eqqq}
\begin{cases}
\tau(f)=\mathcal{R}(f,\psi),\\
\slashed{D}^f\psi=0.
\end{cases}
\end{align}
Here, $\tau(f)=\textup{tr}\nabla(df)=\big(\nabla_{e_\alpha}(df)\big)(e_\alpha)$ is the tension of $f$ and $\mathcal{R}(f,\psi)$ is given by
\begin{align*}
\mathcal{R}(f,\psi)=\frac{1}{2}(\psi^i,e_\alpha\cdot \psi^j)R^{TN}(\frac{\partial}{\partial x_i}\circ f,\frac{\partial}{\partial x_j}\circ f)df(e_\alpha)
\end{align*}
for $\psi=\psi^i\otimes(\frac{\partial}{\partial x_i}\circ f)$. Moreover, $(.,.)$ denotes the real part of the natural hermitian inner product of $\Sigma M$.

Obvious examples $(f,\psi)$ for Dirac-harmonic maps are the following: $f$ is a harmonic map and $\psi=0$, $f$ is a constant map and $\psi\in\textup{ker}(\slashed{D})$ is a harmonic spinor. In that sense, Dirac-harmonic maps generalize the subject of harmonic maps and harmonic spinors.

Results concerning the regularity of Dirac-harmonic maps have been achieved in \cite{Dhmaps,Dhmapsregt,Dhmapsregt2,Dhmapsregt3,Dhmapsregt4,Dhmapsregt5,Dhmapsregt6,Dhmapsregt7,Dhmapsregt8} (mainly in the case that $M$ is $2$-dimensional, since then the functional is conformally invariant).

\subsection{The heat flow for Dirac-harmonic maps}
Apart from the obvious examples explained above, not many concrete examples for Dirac-harmonic maps are known. For a general overview we refer to the discussion in \cite[Section 2]{AmmGin}. First examples for uncoupled Dirac-harmonic maps (i.e., the mapping part is harmonic) are constructed in \cite[Proposition 2.2]{Dhmaps}. Other examples can be found in \cite{explex},\cite{AmmGin}. For coupled Dirac-harmonic maps (i.e., the mapping part is not harmonic) even less is known \cite{explex},\cite{exafterjostmozhu}. 

With the aim to get a general existence program for Dirac-harmonic maps, the \textit{heat flow for Dirac-harmonic maps},
\begin{numcases}{}
\partial_t u= \tau(u)-\mathcal{R}(u,\psi) &$\mathit{on }\text{ } (0,T)\times M$,\label{dhmhf1}
\\
\slashed{D}^u\psi=0 & $\mathit{on }\text{ }[0,T]\times M$,\label{dhmhf2}
\end{numcases}
was introduced in \cite{ChenJostSunZhu}. In the case that $M$ has non-empty boundary, short time existence (and uniqueness) of \eqref{dhmhf1}--\eqref{dhmhf2} was shown in \cite{ChenJostSunZhu} under the presence of certain boundary conditions. Moreover, the existence of a \textit{global} weak solution of \eqref{dhmhf1}--\eqref{dhmhf2} was obtained in \cite{globalweaksol} (again for certain boundary conditions) with some existence results for Dirac-harmonic maps as an application.

At this point we want to mention another approach, considered by Volker Branding in his PhD thesis \cite{DisBranding}, where he studied the evolution equations for so-called regularized Dirac-harmonic maps.

\subsection{Main result and overview of the proof}
Our main result is the short time existence of the heat flow for Dirac-harmonic maps on closed (i.e., compact and without boundary) manifolds.

\begin{theorem}\label{thereom short time existence} Let $M$ be a closed $m$-dimensional Riemannian spin manifold, $m\equiv 0,1,2,4 \text{ }(\textup{mod } 8)$, and $N$ a closed Riemannian manifold of arbitrary dimension. Let $u_0\in C^{2+\alpha}(M,N)$ for some $0<\alpha<1$ with $\textup{dim}_\mathbb{K}\textup{ker}(\slashed{D}^{u_0})=1$, where 
	\begin{align*}
	\mathbb{K}=\begin{cases}
	\mathbb{C} &\text{if }m\equiv 0,1 \text{ }(\textup{mod } 8),\\
	\mathbb{H} &\text{if }m\equiv 2,4 \text{ }(\textup{mod } 8).\\
	\end{cases}
	\end{align*}

	Moreover, let $\psi_0\in\textup{ker}(\slashed{D}^{u_0})$ with $\|\psi_0\|_{L^2}=1$. Then there exists $T>0$ and a solution $(u_t,\psi_t)_{t\in[0,T]}$,
	\[u\in C^{1,2,\alpha}((0,T)\times M,N),\]
	\[\psi_t\in\textup{ker}(\slashed{D}^{u_t})\hspace{3em}\forall t\in[0,T],\]
	of 
	\begin{align}
	\label{eq def DHMHFmodif}\begin{cases}
	\partial_t u= \tau(u)-\mathcal{R}(u,\psi) &\text{ on }(0,T)\times M,\\
	\slashed{D}^u\psi=0&\text{ on }[0,T]\times M,\\
	\textup{dim}_\mathbb{K}\textup{ker}(\slashed{D}^{u_t})=1 &\text{ for all } t\in[0,T],\\
	\|\psi_t\|_{L^2}=1&\text{ for all } t\in[0,T],\\
	u|_{t=0}=u_0,\\
	\psi|_{t=0}=\psi_0.
	\end{cases}
	\end{align}
	Furthermore, if we are given any $T>0$ and a solution $(u_t,\psi_t)_{t\in[0,T]}$ of \eqref{eq def DHMHFmodif} with $u\in C^{1,2,\alpha}((0,T)\times M,N)$, then this solution is unique up to multiplication of the $\psi_t$ with elements of $\mathbb{K}$ whose norm is equal to one.
\end{theorem}

Here, the space $C^{1,2,\alpha}((0,T)\times M,N)$ is defined as follows. First, $C^{1,2,\alpha}((0,T)\times M)$ is the space of functions $u\colon (0,T)\times M\to\mathbb{R}$ s.t. $\sup_{t\in(0,T)}\|u(t,.)\|_{C^{2+\alpha}(M)}<\infty$ and $\sup_{x\in M}\|u(.,x)\|_{C^{1+\frac{\alpha}{2}}(0,T)}<\infty$, c.f. \cite{BDPhil}. Embedding $N$ isometrically into some $\mathbb{R}^q$ we define $C^{1,2,\alpha}((0,T)\times M,N)$ to be the space of all maps $u\colon (0,T)\times M\to N$ s.t. the component functions of $u\colon (0,T)\times M\to N\hookrightarrow \mathbb{R}^q$ belong to $C^{1,2,\alpha}((0,T)\times M)$. Note that every $u\in C^{1,2,\alpha}((0,T)\times M)$ can be continuously extended to $[0,T]\times M$, hence the requirement $u|_{t=0}=u_0$ in \eqref{eq def DHMHFmodif} makes sense.\newline

We want to remark that from our construction of the spinor part $\psi=\psi(u)$ of the solution we will get that $\psi(u)$ depends Lipschitz continuously on $u$ (in the sense of the estimates we derive in Lemma \ref{lemma zentrale Spinorabschaetzung}).\newline

For the existence of initial values we expect something like this: if $M$ is $2$-dimensional and $f\colon M\to N$ is a map with non-vanishing index $\textup{ind}_{f^*TN}(M)\neq 0$ (c.f. Remark \ref{remark index}), then for generic metrics on $N$ it holds that $\textup{dim}_\mathbb{H}\textup{ker}(\slashed{D}^{f})=1$.\newline

In the following, we give an overview of the proof of Theorem \ref{thereom short time existence}. To show short time existence we use the general strategy that was used in \cite{ChenJostSunZhu}, i.e., we first solve the constraint equation \eqref{dhmhf2} for any homotopy of the initial value $u_0$, then we take the solution of the constraint equation and plug it into \eqref{dhmhf1}. After that we use a contraction argument to solve \eqref{dhmhf1} and get the mapping part of the solution.

For the contraction argument we will isometrically embed $N$ into some $\mathbb{R}^q$, rewrite \eqref{dhmhf1} as a heat equation in $\mathbb{R}^q$, and then solve this rewritten equation. However, we will solve the constraint equation \eqref{dhmhf2} in $N$. Note that in \cite{ChenJostSunZhu}, also the constraint equation was rewritten and solved as an equation in $\mathbb{R}^q$.

Clearly we can't solve $\slashed{D}^u\psi=0$ uniquely in the absence of a boundary. However, we can achieve the following: we start with a $1$-dimensional kernel, $\textup{dim}_\mathbb{K}\textup{ker}(\slashed{D}^{u_0})=1$. Then we show that for homotopies of $u_0$ the kernel will stay $1$-dimensional for small times, $\textup{dim}_\mathbb{K}\textup{ker}(\slashed{D}^{u_t})=1$. (This is the only place where the restrictions on the dimension of $M$ will play a role.) Then we impose the additional constraint $\|\psi_t\|_{L^2}=1$ to deduce that we can uniquely solve $\slashed{D}^u\psi=0$ up to multiplication with elements of $\mathbb{K}$ whose norm is equal to one. Now observe that $\mathcal{R}(u,\psi)$ is \textit{invariant} under multiplication of $\psi$ with elements of $\mathbb{K}$ that have norm one. Because of this we can use a contraction argument to show that the mapping part of the solution is in fact unique.

To make the contraction argument work, we need to estimate the solution $\psi=\psi(u)$ of $\slashed{D}^u\psi=0$ in terms of $u$. More precisely, we will construct one such solution and derive estimates for it. To that end, we start with an initial value $\psi_0=\psi(u_0)\in\textup{ker}({\slashed{D}^{u_0}})$. Given a homotopy of $u_0$, we then define $\sigma(u_t)\in\Gamma(\Sigma M\otimes u_t^*TN)$ by identifying the bundles $u_0^*TN$ and $u_t^*TN$ via parallel transport in $N$ along the unique shortest geodesics connecting $u_0(x)$ and $u_t(x)$, $x\in M$. Note that while $\sigma(u_t)$ is in general not in the kernel of $\slashed{D}^{u_t}$, it still has some non-trivial part in the kernel. Hence the projection $\psi(u_t)$ of $\sigma(u_t)$ onto $\textup{ker}(\slashed{D}^{u_t})$ is non-zero. (In particular, we can normalize $\psi(u_t)$ s.t. $\|\psi(u_t)\|_{L^2}=1$.) Writing the projection as a resolvent integral
\[\psi(u_t)=\int_{\gamma}(\mu I-\slashed{D}^{u_t})^{-1}\sigma(u_t)\,d\mu\]
combined with estimates for Dirac operators along maps (which we will derive in Section \ref{section estimates}) we will deduce the necessary estimates for $\psi(u_t)$.

\subsection*{Acknowledgments}I would like to thank Bernd Ammann and Helmut Abels for their ongoing support and many fruitful discussions. I am also grateful to Uli Bunke for his valuable input. My work was supported by the DFG Graduiertenkolleg GRK 1692 ``Curvature, Cycles, and Cohomology''.

\newpage

\section{Preliminaries}

\subsection{Elliptic $W^{k+1}_p$-regularity for Dirac operators along non-smooth maps}

Elliptic $W^{k+1}_p$-regularity for Dirac operators of smooth Dirac bundles is well known. It follows from the mapping properties of pseudo-differential operators with smooth coefficients or it can be shown directly as e.g. in \cite[Theorem 3.2.3]{AmmHabil}. For Dirac operators of non-smooth Dirac bundles it is less known (just as the mapping properties of pseudo-differential operators with non-smooth coefficients are less known). 

In this section we will prove elliptic $W^{k+1}_p$-regularity for Dirac operators along $C^{k+1}$-maps. As a corollary we deduce basic facts about the spectrum of such operators. 

Given $f\in C^1(M,N)$, the Dirac operator along $f$ is an elliptic first order differential operator and formally self-adjoint with respect to the $L^2$-inner product. We view $\slashed{D}^f$ as a bounded densely defined self-adjoint operator
\[\slashed{D}^f\colon\Gamma_{W^1_2}(\Sigma M\otimes f^*TN)\rightarrow\Gamma_{L^2}(\Sigma M\otimes f^*TN).\]
Note that if $f\in C^k(M,N)$, then $f^*TN$ is a $C^k$-vector bundle. Hence we can define $\Gamma_{W^l_p}$ for $l=0,1,\ldots,k$.

\begin{lemma}[Elliptic $W^{k+1}_p$-regularity]\label{lemma ellipticregnonsmooth} Let $M$ be a closed Riemannian spin manifold and $N$ a closed Riemannian manifold, both of arbitrary dimension. Let $f\in C^{k+1}(M,N)$, $k\in\mathbb{N}_{\ge 0}$, and $2\le p<\infty$. Moreover let $\lambda\in\mathbb{C}$ be arbitrary. If
	\[(\lambda I-\slashed{D}^f)\psi=\varphi\]
	for $\psi\in \Gamma_{W^1_2}(\Sigma M\otimes f^*TN)$, $\varphi\in \Gamma_{W^k_p}(\Sigma M\otimes f^*TN)$, then $\psi\in \Gamma_{W^{k+1}_p}(\Sigma M\otimes f^*TN)$ and
	\[\|\psi\|_{W^{k+1}_p}\le C\big(\|\varphi\|_{W^k_p}+\|\psi\|_{L^p}\big)\]
	where $C=C(f,\Sigma M)>0$ is independent of $\psi$, $\varphi$.
\end{lemma}
The basic idea of the proof is to approximate both $\slashed{D}^f$ and the bundle $\Sigma M\otimes f^*TN$ by smooth objects.

\begin{proof}[Proof of Lemma \ref{lemma ellipticregnonsmooth}]First we show the lemma for $\lambda=0$. Given $f\in C^{k+1}(M,N)$ we choose $g\in C^\infty(M,N)$ with $d^N(f(x),g(x))<c$ for all $x\in M$ where $0<c<\frac12\textup{inj}(N)$. In particular we can connect $g(x)$ and $f(x)$ by a unique shortest geodesic of $N$ for every $x\in M$. The parallel transport in $N$ along these geodesics induces $C^{k+1}$-isomorphisms of vector bundles
	\[P\colon g^*TN\to f^*TN,\]
	\[P\colon \Sigma M\otimes g^*TN\to \Sigma M\otimes f^*TN.\]	
	We also get induced isomorphisms of Banach spaces
	\[P\colon \Gamma_{W^{l}_p}(\Sigma M\otimes g^*TN)\to \Gamma_{W^{l}_p}(\Sigma M\otimes f^*TN)\]
	for $l=0,1,\ldots,k+1$.	We consider
	\[G:=P^{-1}\slashed{D}^fP-\slashed{D}^g.\]
	Note that $G$, acting on sections of $\Sigma M\otimes g^*TN$, is a differential operator of order zero.	Heuristically this is the case because in the definition of $G$ the difference of the ordinary Dirac operators $\slashed{D}$ acting on $\Sigma M$ cancel out and we are left with the difference of two covariant derivatives. A covariant derivative has the identity as principal symbol, hence the difference of two covariant derivatives has zero as principal symbol. Therefore $G$ is of order zero. To make this precise we set 
	\[\nabla:=\nabla^{g^*TN}, \hspace{3em} \tilde{\nabla}:=P^{-1}\nabla^{f^*TN}P\]
	(note that $\tilde{\nabla}$ is a (non-smooth) covariant derivative on $g^*TN$). Moreover we choose local frames $(s_j)$ and $(\psi^i)$ of $g^*TN$ and $\Sigma M$, respectively. Given a section $\psi$ of $\Sigma M\otimes g^*TN$ we write
	\[\psi=\lambda^j_i (\psi^i\otimes s_j)\]
	The local formula for Dirac operators along maps yields
	\begin{align*}
	G\psi& = \lambda^j_i (e_\alpha \cdot\psi^i)\otimes \Big(\tilde{\nabla}_{e_\alpha}s_j-\nabla_{e_\alpha}s_j\Big)\\
	&=\lambda^j_i (e_\alpha \cdot\psi^i)\otimes\Big(\tilde{\omega}^l_j(e_\alpha)s_l -\omega^l_j(e_\alpha)s_l\Big)\\
	&= \lambda^j_i(\tilde{\omega}^l_j(e_\alpha)-\omega^l_j(e_\alpha))e_\alpha \cdot (\psi^i\otimes s_l).
	\end{align*}
	From this it is easy to see that $G$ is a differential operator of order zero with $C^k$-coefficients. In particular $G$ extends to a bounded linear map
	\begin{align*}
	G\colon \Gamma_{W^{l}_p}(\Sigma M\otimes g^*TN)\to \Gamma_{W^{l}_p}(\Sigma M\otimes g^*TN),
	\end{align*}
	for $l=0,1,\ldots,k$.
	Now assume that we have
	\[\slashed{D}^f\tilde{\psi}=\tilde{\varphi}\]
	for some $\tilde{\psi}\in \Gamma_{W^1_2}$ and $\tilde{\varphi}\in\Gamma_{W^k_p}$. This is equivalent to
	\[\slashed{D}^g\psi =\varphi - G\psi\]
	where $\psi:=P^{-1}\tilde{\psi}\in\Gamma_{W^1_2}$, $\varphi:=P^{-1}\tilde{\varphi}\in\Gamma_{W^k_p}$. Using the elliptic $W^{k+1}_p$-regularity smooth Dirac bundles \cite[Theorem 3.2.3]{AmmHabil} and a standard bootstrap argument we get $\psi\in \Gamma_{W^{k+1}_p}$ and the existence of some $C=C(\Sigma M)>0$ s.t.
	\begin{align*}
	\|\psi\|_{W^{k+1}_p}\le C(\|\psi\|_{L^p}+\|\varphi\|_{W^k_p}).
	\end{align*}
	Since $P$ is an isomorphism on the Sobolev spaces this implies
	\begin{align*}
	\|\tilde{\psi}\|_{W^{k+1}_p}\le\tilde{C}(\|\tilde{\psi}\|_{L^p}+\|\tilde{\varphi}\|_{W^k_p}).
	\end{align*}
	We have shown the lemma for $\lambda=0$. If $\lambda\neq 0$, then we use the case $\lambda=0$ and a bootstrap argument.
\end{proof}

\begin{corollary}[Spectral properties]\label{cor spectral properties}Assume that $M$ and $N$ are closed. Let $f\in C^1(M,N)$. Given any element $\mu$ of the resolvent set of $\slashed{D}^f$ it holds that the resolvent $R(\mu,\slashed{D}^f)\colon \Gamma_{L^2}\to \Gamma_{L^2}$ of $\slashed{D}^f\colon \Gamma_{W^1_2}\to \Gamma_{L^2}$ is bounded as a map
	\[R(\mu,\slashed{D}^f)\colon \Gamma_{L^2}\to \Gamma_{W^1_2}.\] 
	In particular $\slashed{D}^f$ has compact resolvent. Hence the spectrum $\textup{spec}(\slashed{D}^f)$ of $\slashed{D}^f$ is equal to its point spectrum and $\textup{spec}(\slashed{D}^f)\subset \mathbb{\mathbb{R}}$ is discrete.
\end{corollary}

\subsection{Quaternionic structures on spinor bundles}\label{section quat str}
In this section we collect and recall some facts about quaternionic structures on spinor bundles.

Let $m\equiv 2,3,4 \text{ }(\textup{mod } 8)$. Let $\rho\colon \mathbb{C}l_m\rightarrow \textup{End}_{\mathbb{C}}(\Sigma_m)$ be an irreducible complex algebra representation of the complex Clifford algebra $\mathbb{C}l_m$. By \cite[p. 31]{Fri} and \cite[Theorem 2.2.2.]{DisHer} there exists a quaternionic structure $j\colon \Sigma_m\rightarrow \Sigma_m$ on $\Sigma_m$ (i.e., $j$ is an $\mathbb{R}$-linear map with $j^2=-\textup{id}_V$ and $j(iv)=-ij(v)$ for all $v\in V$) s.t.
\begin{align}\label{eq z35}
j\circ \rho(x)=\rho(x)\circ j
\end{align}
for all $x\in\mathbb{R}^m\subset\mathbb{C}^m\subset \mathbb{C}l_m$. In particular, the complex vector spaces $\Sigma_m$ turn into quaternionic vector spaces (i.e., right $\mathbb{H}$-modules).

Now let $M$ be a $m$-dimensional Riemannian spin manifold with spin structure $\textup{Spin}(M)$. Then every fiber of the (complex) spinor bundle $\Sigma M=\textup{Spin}(M)\times_\rho \Sigma_m$ turns into a quaternionic vector space by defining
\[[p,v]h:=[p,vh]\]
for all $p\in \textup{Spin}(M)$, $v\in \Sigma_m$, and $h\in\mathbb{H}$. Note that this is well-defined because of \eqref{eq z35}. Moreover, given a manifold $N$ and $f\in C^1(M,N)$, every fiber of $\Sigma M\otimes_{\mathbb{R}}f^*TN$ turns into a quaternionic vector space by defining
\[(a\otimes b)h:=(ah)\otimes b\]
for all $a\in \Sigma_xM$, $b\in (f^*TN)_x$, and $h\in\mathbb{H}$. We have the following proposition.

\begin{proposition}Let $f\in C^1(M,N)$. Then it holds that
	\[\slashed{D}^f(\varphi h)=(\slashed{D}^f\varphi)h\]
	for all $\varphi\in\Gamma_{C^1}(\Sigma M\otimes_{\mathbb{R}}f^*TN)$ and all $h\in\mathbb{H}$. In particular, all the eigenspaces of $\slashed{D}^f$ are quaternionic vector spaces.
\end{proposition}

The construction of the natural hermitian inner product on $\Sigma M$ (see e.g. \cite{Hij}) together with the fact that it is unique up to multiplication with positive constants yields the following lemma.
\begin{lemma}\label{lemma bundle metric inv. unit quat}The real part $(.,.)$ of the natural hermitian inner product on $\Sigma M$ is invariant under multiplication by unit quaternions, i.e., it holds that
	\[(\varphi_1h,\varphi_2h)=(\varphi_1,\varphi_2)\]
	for all $\varphi_1,\varphi_2\in\Sigma_xM$, $x\in M$, $h\in S^3\subset\mathbb{C}^2=\mathbb{H}$.
\end{lemma}

\newpage

\section{Setup for the contraction argument}\label{section solution space}

In this section the setup for the contraction argument is developed. After we have stated the precise setting, we will take care of the constraint equation \eqref{dhmhf2} in Section \ref{section constraint}.

\subsection{Translation of equation \eqref{dhmhf1} into $\mathbb{R}^q$}\label{section tubular}
Let $i\colon N\rightarrow \mathbb{R}^q$ be an isometric embedding of $N$ in $\mathbb{R}^q$. In the following we view $N$ as an embedded Riemannian submanifold of $\mathbb{R}^q$ via $i$ and we rewrite the heat flow for Dirac-harmonic maps as an equation in $\mathbb{R}^q$. Let $\delta>0$ s.t. the set
\[N_\delta:=\{y\in\mathbb{R}^q\text{ }|\text{ }d(y,N)<\delta\}\]
is a tubular neighborhood of $N$ in $\mathbb{R}^q$ and there exists a smooth map, called \textit{nearest point projection},
\[\pi\colon N_\delta\rightarrow N\]
s.t. 
\begin{enumerate}
	\item we have $d\pi_xv=pr_{T_xN}v$ for all $x\in N$, $v\in\mathbb{R}^q$,
	\item for every $y\in N_\delta$ it holds that $\pi(y)$ is the unique point of $N$ closest to $y$,
	\item $\pi\colon N_\delta\rightarrow N$ can be extended to a smooth map $\pi\colon \mathbb{R}^q\rightarrow \mathbb{R}^q$ with compact support.
\end{enumerate}
For $A,B\in\{1,\ldots,q\}$ and $z\in\mathbb{R}^q$ we write
\[\pi^A_B(z):=\frac{\partial\pi^A}{\partial z_B}(z)\]
for the $B$-th partial derivative of the $A$-th component function of $\pi\colon\mathbb{R}^q\to\mathbb{R}^q$. Similarly,
\[\pi^A_{BC}(z):=\frac{\partial^2 \pi^A}{\partial z_B\partial z_C}(z).\]

\begin{lemma}
	A tuple $(u,\psi)$ where $u\colon [0,T]\times M\to N$ and $\psi\in\Gamma(pr_2^*\Sigma M\otimes u^*TN)$ is a solution of \eqref{dhmhf1} if and only if it is a solution of
	\begin{align*}
	\partial_tu^A-\Delta u^A &=-\pi^A_{BC}(u)\langle \nabla u^B,\nabla u^C\rangle-\pi^A_B(u)\pi^C_{BD}(u)\pi^C_{EF}(u)(\psi^D,\nabla u^E\cdot\psi^F)
	\end{align*}
on	$(0,T)\times M$, for $A=1,\ldots,q$, where we write $u^A\colon M\rightarrow \mathbb{R}$ for the $A$-th component function of $u\colon [0,T]\times M\rightarrow N\subset\mathbb{R}^q$ and the global sections $\psi^A\in\Gamma(\Sigma M)$ are defined by $\psi=\psi^A\otimes (\partial_A\circ u)$. (Here we write $(\partial_A)_{A=1,\ldots,q}$ for the standard basis of $T\mathbb{R}^q$.) Moreover, $\nabla$ denotes the gradient on $M$ and $\langle.,.\rangle$ the Riemannian metric on $M$.
\end{lemma}
In \cite{ChenJostSunZhu} this lemma was shown by deriving the Euler-Lagrange equations \eqref{eqqq} in the setting provided by the tubular neighborhood. In \cite{JWDissertation} it was shown by direct calculations. For future reference we define
\begin{align*}
F_1^A(u)&:=-\pi^A_{BC}(u)\langle \nabla u^B,\nabla u^C\rangle,\\
F_2^A(u,\psi)&:=-\pi^A_B(u)\pi^C_{BD}(u)\pi^C_{EF}(u)(\psi^D,\nabla u^E\cdot\psi^F).
\end{align*}
Note that for $u\in C^1(M,N)$ and $\psi\in\Gamma(\Sigma M\otimes u^*TN)$ we have
\begin{align*}
\mathcal{R}(u,\psi)|_p&=-F^A_2(u,\psi)|_p \partial_A|_{u(p)},\\
II(du_p(e_\alpha),du_p(e_\alpha))&=-F_1^A(u)|_p\partial_A|_{u(p)}
\end{align*}
for all $p\in M$ where $II$ denotes the second fundamental form of $N\subset\mathbb{R}^q$ and $(e_\alpha)$ is an orthonormal basis of $T_pM$. In particular,
\[\tau(u)|_p=\big(\Delta u^A|_p+F^A_1(u)|_p\big)\partial_A|_{u(p)}.\]
Our notation differs from \cite{ChenJostSunZhu}. We have
\begin{align*}
F_1^A(u)&=\langle \Omega^A_B, du^B\rangle,\\
F_2^A(u,\psi)&=-\langle\tilde{\Omega}^A_B,du^B\rangle
\end{align*}
where on the right hand sides we used the notation of \cite{ChenJostSunZhu}.

\subsection{The fixed point operator and the solution space}\label{section fixed point operator}
For every $T>0$ we denote by $X_T$ the Banach space of bounded maps $[0,T]\to C^1(M,\mathbb{R}^q)$, i.e.,
\begin{align*}
X_T:=B([0,T]; C^1(M,\mathbb{R}^q)),
\end{align*} 
\[\|u\|_{X_T}:=\max_{A=1,\ldots,q}\sup_{t\in[0,T]}\left( \|u^A(t,.)\|_{C^0(M)}+\|\nabla u^A(t,.)\|_{C^0(M)}\right).\]
We choose and fix an initial value for the mapping part $u_0\in C^{2+\alpha}(M,N)$ for some $0<\alpha<1$. Moreover, we define $v_0\in X_T$ by
\[v_0(t,x):=\int_Mp(x,y,t)u_0(y)\,dV(y)\]
where $p$ is the heat kernel of $M$ (see e.g. \cite{Chavel}) and denote by
\[B_R^T(v_0):=\{u\in X_T\text{ }|\text{ }\|u-v_0\|_{X_T}\le R\}\]
the closed ball with center $v_0$ and radius $R$ in $X_T$. Then we set
\begin{align*}
	(Lu)(t,x):=v_0(t,x)+\int_0^t\int_M p(x,y,t-\tau)\left(F_1(u_\tau)(y) +F_2(u_\tau,\psi(u_\tau))(y)\right) \,dV(y)d\tau
\end{align*}
Short time existence then follows from Banach's fixed point theorem after we have shown that $L$ is a contraction on $B_R^T(v_0)$ for $R$ and $T$ small enough. (Of course we have to show some additional things, e.g., that the fixed point takes values in $N$ and has the desired regularity.)

Recalling the strategy of the proof we outlined in the introduction, we first have to solve the constraint equation \eqref{dhmhf2}. (In fact, the $\psi(u)$ in the definition of $L$ will be the solution of the constraint equation.) As we mentioned, we will not transform \eqref{dhmhf2} to $\mathbb{R}^q$ and solve it there, we rather solve it directly in $N$ (in particular, the maps we consider have to be $N$-valued). At this point we run into a technical problem, since the elements of $B_R^T(v_0)$ are $\mathbb{R}^q$-valued. We remedy this by showing that for $R$ and $T$ small enough, every $u\in B_R^T(v_0)$ is $N_\delta$-valued. Hence $\pi\circ u$ is $N$-valued. Then we solve the constraint equation for $\pi\circ u$ instead of $u$ (i.e., we solve $\slashed{D}^{\pi\circ u}=0$ instead of $\slashed{D}^{u}=0$). This does not make a difference, since the fixed point $u_*$ will be $N$-valued, hence $\pi\circ u_*=u_*$.

We also explained in the introduction that to get the necessary estimates for the solution of equation \eqref{dhmhf2}, we will use a construction that joins $u_0(x)$ and $(\pi\circ u_t)(x)$ by a unique shortest geodesic of $N$. To do this, we need the next lemma which states that locally we can bound distances in $N$ by distances in $\mathbb{R}^q$.

\begin{lemma}\label{lemma abstandsvgl umf} Let $N\subset\mathbb{R}^q$ be a closed embedded submanifold of $\mathbb{R}^q$ with the induced Riemannian metric. Denote by $A$ its Weingarten map. Choose $C>0$ s.t. $\|A\|\le C$ where 
	\[\|A\|:=\sup\{\|A_vX\|\text{ }|\text{ }v\in T_p^\bot N, \text{ } X\in T_pN,\text{ }\|v\|=1,\text{ }\|X\|=1, \text{ }p\in N\}.\]
	Then there exists $0<\delta_0<\frac{1}{C}$ s.t. for all $0<\delta\le \delta_0$ and for all $p,q\in N$ with $\|p-q\|_2<\delta$ it holds that
	\[d^N(p,q)\le\frac{1}{1-\delta C}\|p-q\|_2,\]
	where $\|.\|_2$ denotes the Euclidean norm.
\end{lemma}
The above lemma can be proven by e.g. using the Rauch Comparison Theorem for submanifolds \cite[Theorem 4.3. (b)]{Warner}. A detailed proof can be found in \cite{JWDissertation}.

In the following we will make some choices for the constants $\delta,R,$ and $T$ (e.g. to ensure the existence of unique shortest geodesics). At this point it is worth beeing very precise, since the constants will also depend on each other and we want to avoid any unclarity in future arguments. 

It is a standard fact that for every $R>0$ there exists $T=T(R)>0$ s.t.
\begin{align}\label{eq z11}
\|v_0(t,.)-u_0\|_{C^0(M,\mathbb{R}^q)}<R
\end{align}
for all $t\in [0,T]$.

If $R<\frac{\delta}{2}$ and $T=T(R)$ is chosen s.t. \eqref{eq z11} holds, then it holds for every $u\in B_R^T(v_0)\cap\{u|_{t=0}=u_0\}$ that
\[u(t,x)\in N_\delta\]
for all $(t,x)\in[0,T]\times M$. (In particular $\pi\circ u$ is $N$-valued.)

For all $u\in B_R^T(v_0)\cap\{u|_{t=0}=u_0\}$ it holds that
\begin{align}\label{eq z1}
\|(\pi\circ u)(t,x)-u_0(x)\|_2\le 2 \delta
\end{align}	
by the triangle inequality. Now we choose $\varepsilon >0$ with $2\varepsilon <\textup{inj}(N)$. Moreover, let $C>0$ and $\delta_0>0$ be chosen as in Lemma \ref{lemma abstandsvgl umf} and assume
\[\delta<\min\{\frac14\delta_0, \frac{1}{4}\varepsilon (1-\delta_0C) \}.\]
Using equation \eqref{eq z1} we see that for all $u,v\in B_R^T(v_0)\cap\{u|_{t=0}=u_0\}$ it holds that
\begin{align*}
\|(\pi\circ u)(t,x)-(\pi\circ v)(s,x)\|_2&<4\delta<\delta_0
\end{align*}
Therefore, Lemma \ref{lemma abstandsvgl umf} and the choice of $\delta$ yield
\begin{align}\label{eq z4}
d^N((\pi\circ u)(t,x),(\pi\circ v)(s,x))&\le \frac{1}{1-\delta_0C} \|(\pi\circ u)(t,x)-(\pi\circ v)(s,x)\|_2\\
&<\frac{1}{1-\delta_0C}4\delta\nonumber \\
&=\varepsilon\nonumber\\
&<\frac12 \textup{inj}(N).\nonumber
\end{align}
(In particular, we can connect $(\pi\circ u)(t,x)$ and $(\pi\circ v)(s,x)$ by a unique shortest geodesic of $N$.) To summarize, we have chosen constants as follows:\FloatBarrier
\begin{table}[h]
	\begin{center}
		\begin{tabular}{|ll|}
			\hline
			$\varepsilon>0$ &s.t. $2\varepsilon<\textup{inj}(N)$,\\
			$\delta=\delta(\varepsilon)>0$& s.t. $\delta<\min\{\frac14\delta_0, \frac{1}{4}\varepsilon (1-\delta_0C) \}$,\\
			$R=R(\delta,\varepsilon)>0$ & s.t. $R<\frac{\delta(\varepsilon)}{2}$,\\
			$T=T(\delta,\varepsilon, R)>0$ & s.t. \eqref{eq z11} holds\\
			\hline
		\end{tabular}
	\end{center}
	\caption{Choices of constants.}
	\label{tabelle konstanten}
\end{table}\FloatBarrier
where $\delta_0,C>0$ are as in Lemma \ref{lemma abstandsvgl umf}. We have shown that these choices imply 
\[u(t,x)\in N_\delta\]
and
\begin{align}\label{eq u v durch geod verbindbar}
d^N((\pi\circ u)(t,x),(\pi\circ v)(s,x))<\varepsilon<\frac12 \textup{inj}(N)
\end{align}
for all $u,v\in B_R^T(v_0)\cap\{u|_{t=0}=u_0\}$, $x\in M$, $t,s\in[0,T]$. 

In the following, constants appearing in inequalities might depend on $M$, $N$, and $u_0$, but we suppress this dependency in the notation since we view $M$, $N$, and $u_0$ as part of our fixed initial data.

\newpage
\section{The constraint equation}\label{section constraint}
In this section we solve the constraint equation with the strategy outlined in the introduction. Until Section \ref{section kernel projection} we have no restrictions on the dimension of $M$.

Let $u,v\in B_R^T(v_0)\cap\{u|_{t=0}=u_0\}$ and assume that the constants are chosen as in Table \ref{tabelle konstanten}. In the following we denote by $P^{v_s,u_t}=P^{v_s,u_t}(x)$ the parallel transport of $N$ along the unique\footnote{parametrized on $[0,1]$} shortest geodesic from $\pi(v(s,x))$ to $\pi(u(t,x))$. We also denote by $P^{v_s,u_t}$ the induced mappings 
\[(\pi\circ v_s)^*TN\to (\pi\circ u_t)^*TN,\] \[\Sigma M \otimes(\pi\circ v_s)^*TN\rightarrow \Sigma M\otimes(\pi\circ u_t)^*TN,\]and \[\Gamma_{C^1}(\Sigma M \otimes(\pi\circ v_s)^*TN)\rightarrow \Gamma_{C^1}(\Sigma M\otimes(\pi\circ u_t)^*TN).\]

\subsection{Estimates for Dirac operators along maps}\label{section estimates}
As mentioned in the introduction, we will use estimates for Dirac operators along maps to get estimates for the projection onto the kernels of such operators. 

\begin{lemma}\label{lemma Kruemmungsargument} Choose $\varepsilon, \delta, R,$ and $T$ as in Table \ref{tabelle konstanten}. If $\varepsilon>0$ is small enough, then there exists $C=C(R)>0$ s.t.
	\begin{align*}
	\|\left(\big({P^{v_s,u_t}}\big)^{-1}\slashed{D}^{\pi\circ u_t}P^{v_s,u_t}-\slashed{D}^{\pi\circ v_s}\right)\psi(x)\|\le C \|u_t-v_s\|_{C^0(M,\mathbb{R}^q)} \|\psi(x)\|
	\end{align*}
	for all $u,v\in B_R^T(v_0)\cap\{u|_{t=0}=u_0\}$, $\psi\in \Gamma_{C^1}(\Sigma M \otimes (\pi\circ v_s)^*TN)$, $x\in M$, $t,s\in [0,T]$.
\end{lemma}
We formulated the lemma in exactly the way we are going to use it later. However it is obvious from the proof that the assertion of the lemma holds in more general contexts (e.g. for arbitrary maps $f,g\in C^1(M,N)$ that are close enough in $C^0(M,N)$), provided the factors on the right hand side of the inequality are suitably adjusted (e.g. by $C(\|f\|_{C^1},\|g\|_{C^1})\sup_{y\in M}d^N(f(y),g(y))\|\psi(x)\|$).

In the same way, most of the lemmas shown in Section \ref{section constraint} hold in more general situations with essentially the same proofs.

\begin{proof}[Proof of Lemma \ref{lemma Kruemmungsargument}] We write $g:= \pi\circ v_s$, $f:=\pi\circ u_t$, and we define the $C^1$-mapping
	\[F\colon [0,1]\times M \to N\]
	by $F(t,x):=\textup{exp}_{g(x)}(t\textup{exp}_{g(x)}^{-1}f(x))$ where $\textup{exp}$ denotes the exponential map of the Riemannian manifold $N$. Note that $F(0,.)=g$, $F(1,.)=f$, and $t\mapsto F(t,x)$ is the unique shortest geodesic from $g(x)$ to $f(x)$.	We denote by
	\[\mathcal{P}_{t_1,t_2}=\mathcal{P}_{t_1,t_2}(x)\colon T_{F(t_1,x)}N \rightarrow T_{F(t_2,x)}N \]
	the parallel transport in $F^*TN$ w.r.t. $\nabla^{F^*TN}$ (pullback of the Levi-Civita connection on $N$) along the curve $\gamma_x(t):=(t,x)$ from $\gamma_x(t_1)$ to $\gamma_x(t_2)$, $x\in M$, $t_1,t_2\in[0,1]$. In particular,
	\[\mathcal{P}_{0,1}=P^{v_s,u_t}.\]
	Let $\psi\in \Gamma_{C^1}(\Sigma M \otimes g^*TN)$. We have
	\begin{align}\label{eq z2}
	\begin{split}
	&\left(\big(\mathcal{P}_{0,1}\big)^{-1}\slashed{D}^{f}\mathcal{P}_{0,1}-\slashed{D}^{g}\right)\psi\\
	&=(e_\alpha\cdot\psi^i)\otimes\left(\left(\left(\big(\mathcal{P}_{0,1}\big)^{-1}\nabla_{e_\alpha}^{f^*TN}\mathcal{P}_{0,1}\right)-\nabla_{e_\alpha}^{g^*TN}\right)(b_i\circ g)\right)
	\end{split}
	\end{align}
	where $\psi=\psi^i\otimes (b_i\circ g)$, $(b_i)$ is a (smooth) local orthonormal frame of $TN$, $\psi^i$ are local $C^1$-sections of $\Sigma M$, and $(e_\alpha)$ is a (smooth) local orthonormal frame of $TM$. 
	
	Hence, roughly we want to control the difference ``$\nabla_{e_\alpha}^{f^*TN}-\nabla_{e_\alpha}^{g^*TN}$''. The idea to achieve this is to use the fundamental theorem of calculus to get
	\[\text{ }\nabla_{e_\alpha}^{f^*TN}-\nabla_{e_\alpha}^{g^*TN}\approx \int_0^1\nabla_{\frac{\partial}{\partial t}}^{F^*TN}\nabla_{e_\alpha}^{F^*TN}\approx \int_{0}^{1}R^{TN}(dF(\frac{\partial}{\partial t}), dF(e_\alpha))\]
	(of course, this equation does not make sense, it should just sketch the idea of the proof), then use the tensoriality of the curvature tensor and estimate $dF$.
	 
	To that end, we define local $C^1$-sections $\Theta_i$ of $F^*TN$ by 
	\[\Theta_i(t,x):=\mathcal{P}_{0,t}(x)(b_i\circ g)(x).\]
	For each $t\in[0,1]$ we define the functions $T_{ij}(t,.)=T_{ij}^\alpha(t,.)$ by
	\begin{align}\label{eq z3}
	\big(\mathcal{P}_{0,t}\big)^{-1}\big((\nabla_{e_\alpha}^{F^*TN}\Theta_i)(t,x)\big)=\sum_jT_{ij}(t,x)(b_j\circ g)(x).
	\end{align}
	A priori we only know that the $T_{ij}$ are continuous. In the following we will do a few formal calculations and justify them afterwards.
	It holds that
	\begin{align}\label{eq z5}
	\begin{split}
	&\|\left(\left(\big(\mathcal{P}_{0,1}\big)^{-1}\nabla_{e_\alpha}^{f^*TN}\mathcal{P}_{0,1}\right)-\nabla_{e_\alpha}^{g^*TN}\right)(b_i\circ g)(x)\|^2\\
	&=\|\big(\mathcal{P}_{0,1}\big)^{-1}\big((\nabla_{e_\alpha}^{F^*TN}\Theta_i)(1,x)\big)-\big(\mathcal{P}_{0,0}\big)^{-1}\big((\nabla_{e_\alpha}^{F^*TN}\Theta_i)(0,x)\big)\|^2\\
	&=\|\Big(\sum_jT_{ij}(1,x)(b_j\circ g)(x)\Big)-\Big(\sum_jT_{ij}(0,x)(b_j\circ g)(x)\Big)\|^2\\
	&=\sum_j (T_{ij}(1,x)-T_{ij}(0,x))^2\\
	&=\sum_j\left(\int_{0}^{1}\frac{d}{dt}\Big|_{t=r}T_{ij}(t,x)dr\right)^2.
	\end{split}
	\end{align}
	Therefore we want to control the first time-derivative of the $T_{ij}$. Equation \eqref{eq z3} implies that these time-derivatives are related to the curvature of $F^*TN$. More precisely, for all $X\in\Gamma(TM)$ we have
	\begin{align}\label{eq z46}\begin{split}
	&\frac{d}{dt}\Big|_{t=r}\left(\mathcal{P}_{0,t}^{-1}\left(\left(\nabla_X^{F^*TN}\Theta_i\right)(t,x)\right)\right)\\
	&=\frac{d}{dt}\Big|_{t=0}\left(\mathcal{P}_{0,t+r}^{-1}\left(\left(\nabla_X^{F^*TN}\Theta_i\right)(t+r,x)\right)\right)\\
	&=\mathcal{P}_{0,r}^{-1}\frac{d}{dt}\Big|_{t=0}\left(\mathcal{P}_{r,t+r}^{-1}\left(\left(\nabla_X^{F^*TN}\Theta_i\right)(t+r,x)\right)\right)\\
	&=\mathcal{P}_{0,r}^{-1}\left(\left(\nabla_{\frac{\partial}{\partial t}}^{F^*TN}\nabla_X^{F^*TN}\Theta_i\right)(r,x)\right).\end{split}
	\end{align}
	Now we justify the formal calculations \eqref{eq z5} and \eqref{eq z46}. Combining the definition of $\Theta_i$ as parallel transport and a careful examination of the regularity of $F$ we deduce that $\left(\nabla_{\frac{\partial}{\partial t}}^{F^*TN}\nabla_X^{F^*TN}\Theta_i\right)(r,x)$ exists (in the sense that the expression is well-defined in local coordinates). Then \eqref{eq z46} holds. In particular $\mathcal{P}_{0,t}^{-1}\left(\left(\nabla_X^{F^*TN}\Theta_i\right)(t,x)\right)$ is differentiable in $t$. Then \eqref{eq z3} yields that the $T_{ij}$ are differentiable in $t$. Therefore \eqref{eq z5} holds. 
	
	We further get
	\begin{align*}
	\nabla_{\frac{\partial}{\partial t}}^{F^*TN}\nabla_X^{F^*TN}\Theta_i&=R^{F^*TN}(\frac{\partial}{\partial t},X)\Theta_i+\nabla_X^{F^*TN}\nabla_{\frac{\partial}{\partial t}}^{F^*TN}\Theta_i-\nabla^{F^*TN}_{[\frac{\partial}{\partial t},X]}\Theta_i\\
	&=R^{F^*TN}(\frac{\partial}{\partial t},X)\Theta_i\\
	&=R^{TN}(dF(\frac{\partial}{\partial t}), dF(X))\Theta_i,
	\end{align*}
	since $\nabla_{\frac{\partial}{\partial t}}^{F^*TN}\Theta_i=0$ by definition of $\Theta_i$, and $[\frac{\partial}{\partial t},X]=0$.\footnote{For this chain of equations one has to be a little careful, since we argue with the curvature of $F^*TN$, but $F$ is only a $C^1$-mapping. However, all the expressions exist (e.g. in the sense that the exist in local coordinates) and all the equalities hold. }
	
	This implies
	\begin{align*}
	\sum_j \left(\frac{d}{dt}\Big|_{t=r}T_{ij}(t,x)\right)^2&=\|\frac{d}{dt}\Big|_{t=r}\left(\mathcal{P}_{0,t}^{-1}\left(\left(\nabla_{e_\alpha}^{F^*TN}\Theta_i\right)(t,x)\right)\right)\|^2\\
	&=\|\left(\nabla_{\frac{\partial}{\partial t}}^{F^*TN}\nabla_{e_\alpha}^{F^*TN}\Theta_i\right)(r,x)\|^2\\
	&=\|R^{TN}(dF_{(r,x)}(\frac{\partial}{\partial t}), dF_{(r,x)}(e_\alpha))\Theta_i(r,x)\|^2\\
	&\le C_1 \|dF_{(r,x)}(\frac{\partial}{\partial t})\|^2\|dF_{(r,x)}(e_\alpha)\|^2
	\end{align*}
	where $C_1$ depends only on $N$.
	
	Therefore it remains to estimate $\|dF_{(r,x)}(\frac{\partial}{\partial t})\|$ and $\|dF_{(r,x)}(e_\alpha)\|$ appropriately. We have
	\begin{align*}
	dF_{(r,x)}(\frac{\partial}{\partial t}\Big|_{(r,x)})=\frac{d}{dt}\Big|_{t=r}(\textup{exp}_{g(x)}(t\textup{exp}_{g(x)}^{-1}f(x))=c'(r)
	\end{align*}
	where $c(t):=\textup{exp}_{g(x)}(t \textup{exp}_{g(x)}^{-1}f(x))$ is a geodesic of $N$. In particular $c'$ is parallel along $c$ and thus $\|c'(r)\|=\|c'(0)\|=\|\textup{exp}_{g(x)}^{-1}f(x)\|$. Therefore we get
	\begin{align*}
	\|dF_{(r,x)}(\frac{\partial}{\partial t}\Big|_{(r,x)})\|=\|\textup{exp}_{g(x)}^{-1}f(x)\|\le d^N(g(x),f(x))\le C_2\|u_t-v_s\|_{C^0(M,\mathbb{R}^q)}
	\end{align*}
	where we used \eqref{eq z4} and the (global) Lipschitz continuity of $\pi$. Moreover, there exists some $C_3(R)>0$ s.t. $\|dF_{(r,x)}(e_\alpha)\|\le C_3(R)$ for all $(r,x)\in [0,1]\times M$. (This is not hard to show, but a bit tedious.) We have shown
	\begin{align*}
	\sum_j \left(\frac{d}{dt}\Big|_{t=r}T_{ij}(t,x)\right)^2\le C_1 C_2^2 C_3(R)^2 \|f-g\|_{C^0(M,\mathbb{R}^q)}^2=C_4(R)\|f-g\|_{C^0(M,\mathbb{R}^q)}^2
	\end{align*}
	for all $(t,x)$. Combining this with \eqref{eq z2} and \eqref{eq z5} yields the lemma.
\end{proof}

\newpage
\subsection{Estimates for the parallel transports}
In this section we obtain estimates for the parallel transports which will be used later.

\begin{lemma}\label{lemma parallel transport deviation} Choose $\varepsilon, \delta, R,$ and $T$ as in Table \ref{tabelle konstanten}. If $\varepsilon >0$ is small enough, then there exists $C=C(\varepsilon)>0$ s.t.
	\[\|P^{v_s,u_0}P^{u_t,v_s}P^{u_0,u_t}Z-Z\|\le C\|u_t-v_s\|_{C^0(M,\mathbb{R}^q)} \|Z\|\]
	for all $Z\in T_{u_0(x)}N$, $x\in M$, $s,t\in[0,T]$, $u,v\in B_R^T(v_0)\cap\{u|_{t=0}=u_0\}$.
\end{lemma}

\begin{proof}We fix $x,s,t,u,v$ and write $y:=(\pi\circ v_s)(x)$, $z:=(\pi \circ u_t)(x)$. Moreover, we denote by $\gamma_i\colon [0,1]\to N$ the unique shortest geodesics of $N$ with
	\[\gamma_1(0)=\gamma_3(1)=u_0(x), \text{ }\gamma_1(1)=\gamma_2(0)=z, \text{ }\gamma_2(1)=\gamma_3(0)=y.\]
	Furthermore, we define $c:=\gamma_3*\gamma_2*\gamma_1$, i.e., $c$ is the curve obtained by first following $\gamma_1$, then $\gamma_2$, and then $\gamma_3$. Finally, we write $P^c$ for the induced parallel transport of $N$ along $c$. (Hence $P^c=P^{v_s,u_0}P^{u_t,v_s}P^{u_0,u_t}$.)
	
	We consider the (well-defined) geodesic variation
	\begin{align*}
	\alpha\colon [0,1]\times[0,1]\rightarrow N,\hspace{3em}(s,t)\mapsto \textup{exp}_{u_0(x)}(t\textup{exp}_{u_0(x)}^{-1}\gamma_2(s)),
	\end{align*}
	where $\textup{exp}$ is the exponential map of $N$.
	
	We choose an arbitrary $Z\in T_{u_0(x)}N$. In the following we derive a formula that relates $P^{c}Z-Z$ to an integral over $R^{TN}(\frac{\partial}{\partial s}\alpha(s,t),\frac{\partial}{\partial t}\alpha(s,t))$ with a strategy inspired by \cite[Section 7]{Yang}. This formula is closely related to the general fact that  ``deviation of parallel transport from the identity $\approx$ curvature $\cdot$ enclosed area''.
	
	Denote by $t\mapsto Z(t)$ the parallel vector field along $\gamma_1$ with $Z(0)=Z$. For every $t\in[0,1]$ let $s\mapsto Z(s,t)$ be the parallel vector field along $s\mapsto \alpha(s,t)$ with $Z(0,t)=Z(t)$. In particular we have
	\[P^{\gamma_2*\gamma_1}Z=Z(1,1).\]
	
	Let $(E_0,\ldots,E_n)$ be an orthonormal basis of $T_{u_0(x)}N$. Analogously, we construct $E_i(s,t)\in T_{\alpha(s,t)}N$ s.t. $E_i(0,0)=E_i$, $t\mapsto E_i(1,t)$ is parallel along $t\mapsto \alpha(1,t)$, and $s\mapsto E_i(s,t)$ is parallel along $s\mapsto \alpha(s,t)$ for every $t\in[0,1]$.
	
	We write $Z(s,t)=Z^i(s,t)E_i(s,t)$, i.e., $Z^i(s,t)=\langle Z(s,t), E_i(s,t)\rangle$ (here, $\langle.,.\rangle$ denotes the Riemannian metric on $N$). It holds that
	\begin{align*}
	\frac{d}{dt}\Big|_{t=t_0} Z^i(1,t)=\frac{d}{dt}\Big|_{t=t_0}\langle Z(1,t), E_i(1,t)\rangle=\langle \frac{D}{dt}\Big|_{t=t_0}Z(1,t), E_i(1,t_0)\rangle
	\end{align*}
	and
	\begin{align*}
	&\frac{d}{ds}\Big|_{s=s_0} \langle \frac{D}{dt}\Big|_{t=t_0}Z(s,t), E_i(s,t_0)\rangle\\
	&=\langle\frac{D}{ds}\Big|_{s=s_0} \frac{D}{dt}\Big|_{t=t_0}Z(s,t), E_i(s_0,t_0)\rangle\\
	&=\langle R^{TN}\Big(\frac{\partial}{\partial s}\Big|_{s=s_0}\alpha(s,t_0), \frac{\partial}{\partial t}\Big|_{t=t_0}\alpha(s_0,t)\Big)Z(s_0,t_0)\\
	&\hphantom{=}+ \frac{D}{dt}\Big|_{t=t_0}\frac{D}{ds}\Big|_{s=s_0}Z(s,t), E_i(s_0,t_0)\rangle\\
	&=\langle R^{TN}\Big(\frac{\partial}{\partial s}\Big|_{s=s_0}\alpha(s,t_0), \frac{\partial}{\partial t}\Big|_{t=t_0}\alpha(s_0,t)\Big)Z(s_0,t_0), E_i(s_0,t_0)\rangle.
	\end{align*}
	Noting that $Z=Z(0,0)=Z(s,0)$ and $E_i=E_i(0,0)=E_i(s,0)$ for all $s\in[0,1]$ (since $s\mapsto\alpha(s,0)$ is constant) we get
	
	\begin{align*}
	P^{c}Z-Z&=P^{\gamma_3}\left(P^{\gamma_2*\gamma_1}Z\right)-Z\\
	&=P^{\gamma_3}\left(Z^i(1,1)E_i(1,1)\right)-Z^i(1,0)E_i(1,0)\\
	&=Z^i(1,1)E_i(1,0)-Z^i(1,0)E_i(1,0)\\
	&=\int_{0}^{1}\langle \frac{D}{dt}Z(1,t), E_i(1,t)\rangle dt E_i\\
	&=\int_{0}^{1}\Big(\langle \frac{D}{dt}Z(1,t), E_i(1,t)\rangle-\langle \underbrace{\frac{D}{dt}Z(0,t)}_{=0}, E_i(0,t)\rangle\Big)dt E_i\\
	&=\int_{0}^{1}\int_{0}^{1}\frac{d}{ds}\langle \frac{D}{dt}Z(s,t), E_i(s,t)\rangle dtdsE_i\\
	&=\int_{0}^{1}\int_{0}^{1}\langle R^{TN}\Big(\frac{\partial}{\partial s}\alpha(s,t), \frac{\partial}{\partial t}\alpha(s,t)\Big)Z(s,t), E_i(s,t)\rangle dtdsE_i.
	\end{align*}
	We have shown that
	\begin{align}\label{eq z8}
	\begin{split}
	P^{c}Z-Z=\Bigg(\int_{0}^{1}\int_{0}^{1}\langle R^{TN}\Big(\frac{\partial}{\partial s}\alpha(s,t), \frac{\partial}{\partial t}\alpha(s,t)\Big)Z(s,t), E_i(s,t)\rangle dtds\Bigg)E_i
	\end{split}
	\end{align} 
	holds for all $Z\in T_{u_0(x)}N$. In the next step we estimate $\|\frac{\partial}{\partial t}\alpha\|$ and $\|\frac{\partial}{\partial s}\alpha\|$. To that end, notice that
	\begin{align*}
	\|\frac{\partial}{\partial t}\alpha(s,t)\|=\|\textup{exp}^{-1}_{u_0(x)}\gamma_2(s)\|\le d^N(u_0(x),\gamma_2(s))<2\varepsilon.
	\end{align*} 
	Therefore it remains to estimate $\|\frac{\partial}{\partial s}\alpha\|$.  For each $s\in[0,1]$ we consider the Jacobi field \[J_s(t):=\frac{\partial}{\partial s}\alpha(s,t).\]
	Equation \eqref{eq z4} and the (global) Lipschitz continuity of $\pi\colon\mathbb{R}^q\rightarrow\mathbb{R}^q$ yield
	\begin{align}\label{eq z9}
	\|J_s(1)\|=\|\gamma_2'(s)\|=\|\textup{exp}^{-1}_zy\|\le d^N(z,y)\le C_1 \|u_t-v_s\|_{C^0(M,\mathbb{R}^q)}
	\end{align}
    Using standard comparison theory for Riemannian manifolds with sectional curvature bounded from above (e.g. \cite[equation (5.5.5) in Theorem 5.5.1]{Jost}) we deduce
	\[\|J_s(t)\|\le \|J_s(1)\|\]
	for all $t\in [0,1]$, provided that $\varepsilon >0$ is small enough. If we combine this with \eqref{eq z8} and \eqref{eq z9} we get
	\begin{align*}
	\|P^{c}Z-Z\|\le C(\varepsilon)\|u_t-v_s\|_{C^0(M,\mathbb{R}^q)}\|Z\|
	\end{align*}
	for all $Z\in T_{u_0(x)}N$.
\end{proof}

The operator norms of the induced maps
\[P^{v_s,u_t}\colon\Gamma_{W^1_p}(\Sigma M \otimes(\pi\circ v_s)^*TN)\rightarrow \Gamma_{W^1_p}(\Sigma M\otimes(\pi\circ u_t)^*TN)\]
are finite. However, we need that these operator norms are uniformly bounded in $v_s$ and $u_t$. To that end we need the following lemma.

\begin{lemma}\label{lemma z1}Choose $\varepsilon, \delta, R,$ and $T$ as in Table \ref{tabelle konstanten}. There exists $C=C(R)>0$ s.t.
	\begin{align*}
	\|\nabla^{(\pi\circ u_t)^*TN}_X \big(P^{v_s,u_t}Z\big)|_x\|\le C\|Z\|_{\Gamma_{C^1}((\pi\circ v_s)^*TN)} \|X\|
	\end{align*}
	for all $X\in T_xM$, $x\in M$, $Z\in\Gamma_{C^1}((\pi\circ v_s)^*TN)$, $s,t\in[0,T]$, $u,v\in B_R^T(v_0)\cap\{u|_{t=0}=u_0\}$.
\end{lemma}
\begin{proof}
	We write $f:=\pi\circ u_t$, $g:=\pi\circ v_s$, $P:=P^{v_s,u_t}$, and moreover $\nabla:=\nabla^{(\pi\circ u_t)^*TN}$. Let $Z\in\Gamma_{C^1}((\pi\circ v_s)^*TN)$, $x\in M$, $X\in T_xM$, and $\gamma\colon (-c,c)\rightarrow M$ a smooth curve parametrized proportionally to arc length with $\gamma(0)=x$, $\gamma'(0)=X$. Let $(E_i(.))$ be a local orthonormal frame around $x$ of $f^*TN$ that is parallel along $\gamma$. Locally we have
	\[P(y)Z(y)=f^i(y)E_i(y)\]
	for suitable functions $f^i$. Then it holds that
	\[\nabla_X\big(PZ\big)|_x=\big(L_Xf^i\big)(x)E_i(x).\]
	In the following, we estimate $(L_Xf^i)E_i$. To that end, we denote by $P^\gamma$ the parallel transport in $TN$ along $f\circ \gamma$ from $f(x)$ to $f(\gamma(\tau))$. We also denote by $P^\gamma$ the parallel transport in $TN$ along $g\circ \gamma$ from $g(x)$ to $g(\gamma(\tau))$. It should always be clear from the context which one we mean.	We calculate
	\begin{align*}
	&(L_Xf^i)(x)E_i(x)\\
	&=\lim_{\tau\to 0}\frac{f^i(\gamma(\tau))-f^i(x)}{\tau} E_i(x)\\
	&=\lim_{\tau\to 0}\frac{f^i(\gamma(\tau))E_i(x)-f^i(x)E_i(x)}{\tau}\\
	&=\lim_{\tau\to 0}\frac{(P^\gamma)^{-1}\big(f^i(\gamma(\tau))P^\gamma E_i(x)\big)-PZ(x)}{\tau}\\
	&=\lim_{\tau\to 0}\frac{(P^\gamma)^{-1}\big(f^i(\gamma(\tau))E_i(\gamma(\tau))\big)-PZ(x)}{\tau}\\
	&=\lim_{\tau\to 0}\frac{(P^\gamma)^{-1}PZ(\gamma(\tau))-PZ(x)}{\tau}\\
	&=\lim_{\tau\to 0}\frac{(P^\gamma)^{-1}PZ(\gamma(\tau))-(P^{\gamma})^{-1}PP^\gamma Z(x)+(P^{\gamma})^{-1}PP^\gamma Z(x)-PZ(x)}{\tau}\\
	\end{align*}
	We will show
	\begin{align}\label{eq zzz10}
	\|\big((P^\gamma)^{-1}Z(\gamma(\tau))\big)-Z(x)\|\le \tau \|X\|\|Z\|_{\Gamma_{C^1}(g^*TN)}
	\end{align}
	and 
	\begin{align}\label{eq zzz11}
	\begin{split}
	&\|P^{-1}(P^\gamma)^{-1}PP^\gamma Z(x)-Z(x)\|\\
	&\le \tau C(R)\|u_t-v_s\|_{C^0(M,\mathbb{R}^q)} \|X\|\|Z(x)\|.
	\end{split}
	\end{align}
	After that the lemma follows easily.
	
	Equation \eqref{eq zzz10} directly follows from the fundamental theorem of calculus and the fact that we can recover a covariant derivative by differentiating its parallel transport. To show \eqref{eq zzz11} we recall that 
	\[P^\Box:=P^{-1}(P^\gamma)^{-1}PP^\gamma\colon T_{g(x)}N\to T_{g(x)}N\]
	is the parallel transport along the following rectangle $\Box$: first we follow $g\circ \gamma$ from $g(x)$ to $g(\gamma(\tau))$. Then we go along the unique shortest geodesic of $N$ connecting $g(\gamma(\tau))$ and $f(\gamma(\tau))$. Afterwards we follow $f\circ g$ from $f(\gamma(\tau))$ to $f(x)$. Finally we go along the unique shortest geodesic of $N$ connecting $f(x)$ and $g(x)$.
	We can estimate $\|P^\Box-\textup{Id}\|$ with the same methods we used to show Lemma \ref{lemma parallel transport deviation}. More precisely, we consider the geodesic variation
	\[\alpha\colon [0,\tau]\times[0,1]\to N,\hspace{3em} (s_1,t_1)\mapsto \textup{exp}_{g(\gamma(s_1))}(t_1\textup{exp}_{g(\gamma(s_1))}^{-1}f(\gamma(s_1))).\]
	By definition, the image of $\alpha$ is the filled rectangle $\Box$. Analogously to the proof of Lemma \ref{lemma parallel transport deviation} (the fact that we consider a rectangle now but before we considered a triangle doesn't change the nature of the argument) we get
	\begin{align*}
		\|P^{\Box}Z-Z\|\le \tau C_1\sup_{(s_1,t_1)\in[0,\tau]\times[0,1]} \|\frac{\partial}{\partial t_1}\alpha(s_1,t_1)\| \sup_{(s_1,t_1)\in[0,\tau]\times[0,1]}\|\frac{\partial}{\partial s_1}\alpha(s_1,t_1)\|
	\end{align*}
	where $C_1$ only depends on the Riemannian manifold $N$. Moreover, by \eqref{eq z4} and the (global) Lipschitz continuity of $\pi$ we have
	\begin{align*}
	\|\frac{\partial}{\partial t_1}\alpha(s_1,t_1)\|=\|\textup{exp}^{-1}_{g(\gamma(s_1))}f(\gamma(s_1))\|\le \frac{1}{1-\delta_0C} \|v_s - u_t\|_{C^0(M,\mathbb{R}^q)}
	\end{align*}
	for all $s_1,t_1$. Since it also holds that
	\begin{align*}
	\|\frac{\partial}{\partial s_1}\alpha(s_1,t_1)\|\le C(R)\|X\|
	\end{align*}
	for all $s_1,t_1$, equation \eqref{eq zzz11} follows.
	
\end{proof}

From Lemma \ref{lemma z1} we directly get the following corollary.

\begin{corollary}\label{cor parallel bdd w1p}Choose $\varepsilon, \delta, R,$ and $T$ as in Table \ref{tabelle konstanten}. For $u,v\in B_R^T(v_0)\cap\{u|_{t=0}=u_0\}$, $s,t\in[0,T]$, the isometries 
	\[P^{v_s,u_t}\colon\Gamma_{L^p}(\Sigma M \otimes(\pi\circ v_s)^*TN)\rightarrow \Gamma_{L^p}(\Sigma M\otimes(\pi\circ u_t)^*TN)\] 
	restrict to isomorphisms of Banach spaces
	\[P^{v_s,u_t}\colon\Gamma_{W^1_p}(\Sigma M \otimes(\pi\circ v_s)^*TN)\rightarrow \Gamma_{W^1_p}(\Sigma M\otimes(\pi\circ u_t)^*TN)\]
	with uniformly bounded operator norm, i.e., there exists $C=C(R,p)$ s.t.
	\[\|P^{v_s,u_t}\|_{L(W^1_p,W^1_p)}\le C\]
	for all $u,v\in B_R^T(v_0)\cap\{u|_{t=0}=u_0\}$, $s,t\in[0,T]$.
\end{corollary}
\newpage

\subsection{The projection onto the kernel}\label{section kernel projection}
When we write $\textup{ker}(\slashed{D}^{\pi\circ u_t})$ in the following we mean the kernel of
\[\slashed{D}^{\pi\circ u_t}\colon\Gamma_{W^1_2}\to \Gamma_{L^2}.\]
In this section we assume $m=\textup{dim}(M)\equiv 0,1,2,4 \text{ }(\textup{mod } 8)$. In Remark \ref{remark dim} below it is explained why we restrict to these dimensions. Note that the dimension of $N$ is still arbitrary.
\begin{lemma}\label{lemma kern lokal konstant} Assume that $\textup{dim}_\mathbb{K}\textup{ker}(\slashed{D}^{u_0})=1$, where
\begin{align*}
\mathbb{K}=\begin{cases}
\mathbb{C} &\text{if }m\equiv 0,1 \text{ }(\textup{mod } 8),\\
\mathbb{H} &\text{if }m\equiv 2,4 \text{ }(\textup{mod } 8).\\
\end{cases}
\end{align*}	
Choose $\varepsilon, \delta, R,$ and $T$ as in Lemma \ref{lemma Kruemmungsargument}. If $R>0$ is small enough, then it holds that \[\textup{dim}_\mathbb{K}\textup{ker}(\slashed{D}^{\pi\circ u_t})=1\]
for all $u\in B_R^T(v_0)\cap \{u|_{t=0}=u_0\}$, $t\in [0,T]$, and there exists $\Lambda=\Lambda(R)>0$ s.t. 
\[\textup{spec}(\slashed{D}^{\pi\circ u_t})\setminus\{0\}\subset \mathbb{R}\setminus(-\Lambda,\Lambda)\]
for all $u\in B_R^T(v_0)\cap \{u|_{t=0}=u_0\}$, $t\in [0,T]$.
\end{lemma}

\begin{proof}
	Since the spectrum of $\slashed{D}^{u_0}$ is a discrete subset of $\mathbb{R}$, we can choose $\tilde{\Lambda}>0$ s.t. $\textup{spec}(\slashed{D}^{u_0})\setminus\{0\}\subset \mathbb{R}\setminus(-\tilde{\Lambda},\tilde{\Lambda})$.	Let $u\in B_R^T(v_0)\cap\{u|_{t=0}=u_0\}$.
	For any $\psi \in \textup{ker}(\slashed{D}^{u_0})\setminus\{0\}$ we write $\tilde{\psi}:=P^{u_0,u_t}\psi$. Using Lemma \ref{lemma Kruemmungsargument} we get
	\begin{align}\label{eq 111}
	\|\slashed{D}^{\pi\circ u_t}\tilde{\psi}\|_{L^2}=\|\left(P^{u_0,u_t}\slashed{D}^{u_0}(P^{u_0,u_t})^{-1}-\slashed{D}^{\pi\circ u_t}\right)\tilde{\psi}\|_{L^2}\le C(R)R\|\tilde{\psi}\|_{L^2}
	\end{align}
	for all $\psi\in\textup{ker}(\slashed{D}^{u_0})\setminus\{0\}$. Hence we can estimate the Rayleigh quotient of $(\slashed{D}^{\pi\circ u_t})^2$ by
	\begin{align*}
	\frac{((\slashed{D}^{\pi\circ u_t})^2\tilde{\psi},\tilde{\psi})_{L^2}}{(\tilde{\psi},\tilde{\psi})_{L^2}}\le C_1(R)R
	\end{align*}
	for all $\psi\in\textup{ker}(\slashed{D}^{u_0})\setminus\{0\}$. Applying the Min-Max principle, we deduce that $(\slashed{D}^{\pi\circ u_t})^2$ has at least one eigenvalue (we count eigenvalues by their $\mathbb{K}$-multiplicity) in the interval $[0,C_1(R)R]$. In particular, $\slashed{D}^{\pi\circ u_t}$ has at least one eigenvalue in $[-C_1(R)R,C_1(R)R]$. Now we set 
	\[\Lambda:=\frac12 \tilde{\Lambda}\]
	and choose $R>0$ so small that $C(R)R<\frac12\Lambda$ and $C_1(R)R<\frac12\Lambda$. Hence we have shown that $\slashed{D}^{\pi\circ u_t}$ has at least one eigenvalue in $[-\Lambda,\Lambda]$. With the same methods we can show that $\slashed{D}^{\pi\circ u_t}$ has precisely one eigenvalue in $[-\Lambda,\Lambda]$. Suppose this is not the case. Choose two eigenvalues $\lambda_1,\lambda_2$ of $\slashed{D}^{\pi\circ u_t}$ in $[-\Lambda,\Lambda]$ with corresponding eigenvectors $\psi_1,\psi_2\in\Gamma_{W^1_2}$. For $\tilde{\psi}_i:={P^{u_0,u_t}}^{-1}\psi_i$ we get as above
	\begin{align*}
	\|(\lambda_i-\slashed{D}^{u_0})\tilde{\psi}_i\|_{L^2}\le C(R)R\|\tilde{\psi}_i\|_{L^2}< \Lambda\|\tilde{\psi}_i\|_{L^2}.
	\end{align*}
	Therefore,
	\begin{align*}
	\frac{\|\slashed{D}^{u_0}\tilde{\psi}_i\|_{L^2}}{\|\tilde{\psi}_i\|_{L^2}}&\le \frac{\|(\lambda_i-\slashed{D}^{u_0})\tilde{\psi}_i\|_{L^2}}{\|\tilde{\psi}_i\|_{L^2}} + \frac{\|\lambda_i \tilde{\psi}_i\|_{L^2}}{\|\tilde{\psi}_i\|_{L^2}}
	\le \Lambda + \Lambda
	=\tilde{\Lambda}.
	\end{align*}
	As before, we conclude that $\slashed{D}^{u_0}$ has at least two eigenvalues in $[-\tilde{\Lambda},\tilde{\Lambda}]$. Because of the choice of $\tilde{\Lambda}$ this is a contradiction to $\textup{dim}_\mathbb{K}\textup{ker}(\slashed{D}^{u_0})=1$.
	
	We have shown that $\slashed{D}^{\pi\circ u_t}$ has precisely one eigenvalue in $[-\Lambda,\Lambda]$. The symmetry of the spectrum of $\slashed{D}^{\pi\circ u_t}$ yields that this eigenvalue has to be zero.\footnote{If $m\not\equiv 3 \text{ }(\textup{mod } 4)$, then the spectrum of $\slashed{D}^f$ is symmetric w.r.t. zero. This can be shown analogously to \cite[Theorem 1.3.7 iv)]{GiSpec}.}
	
	It remains to show that it holds that $\textup{spec}(\slashed{D}^{\pi\circ u_t})\setminus\{0\}\subset \mathbb{R}\setminus(-\Lambda,\Lambda)$ for all $u\in B_R^T(v_0)\cap \{u|_{t=0}=u_0\}$, $t\in [0,T]$. To that end we assume that this is not the case for some $s\in(0,T]$. Then there exists $\mu\in\mathbb{R}\setminus\{0\}$ with $|\mu|\le\Lambda$ and $(\mu-\slashed{D}^{\pi\circ u_s})\psi=0$ for some $\psi\in\Gamma_{W^1_2}\setminus\{0\}$. Again for $\tilde{\psi}:={P^{u_0,u_s}}^{-1}\psi$ we have
	\begin{align*}
	\|(\mu-\slashed{D}^{u_0})\tilde{\psi}\|_{L^2}\le C(R)R\|\tilde{\psi}\|_{L^2}< \Lambda\|\tilde{\psi}\|_{L^2}.
	\end{align*}
    Therefore,
	\begin{align*}
	\frac{\|\slashed{D}^{u_0}\tilde{\psi}\|_{L^2}}{\|\tilde{\psi}\|_{L^2}}&\le \frac{\|(\mu-\slashed{D}^{u_0})\tilde{\psi}\|_{L^2}}{\|\tilde{\psi}\|_{L^2}} + \frac{\|\mu \tilde{\psi}\|_{L^2}}{\|\tilde{\psi}\|_{L^2}}
	\le \Lambda + \Lambda
	=\tilde{\Lambda}.
	\end{align*}
	Additionally, as above we have
	\begin{align*}
	\frac{\|\slashed{D}^{u_0}{P^{u_0,u_s}}^{-1}\varphi\|_{L^2}}{\|{P^{u_0,u_s}}^{-1}\varphi\|_{L^2}}\le \tilde{\Lambda}
	\end{align*}
	for all $\varphi\in\textup{ker}(\slashed{D}^{\pi \circ u_{s}})\setminus\{0\}$.
	As before, we conclude that $\slashed{D}^{u_0}$ has at least $1+1=2$ Eigenvalues in $[-\tilde{\Lambda},\tilde{\Lambda}]$, which is a contradiction.
\end{proof}

\begin{remark}\label{remark dim}Lemma \ref{lemma kern lokal konstant} is the only place where the restrictions on $m$ in Theorem \ref{thereom short time existence} play a role. In the proof of Lemma \ref{lemma kern lokal konstant} we used that the spectrum of the Dirac operator along maps is symmetric, which holds if $m\not\equiv 3,7 \text{ }(\textup{mod } 8)$. The dimensions $m\equiv 5,6 \text{ }(\textup{mod } 8)$ were excluded, since in these dimensions there exists no quaternionic structure on $\Sigma_m$ that commutes with Clifford-multiplication, however, there exists a quaternionic structure on $\Sigma_m$ that \textit{anticommutes} with Clifford-multiplication. This yields that the kernel of $\slashed{D}^f$ is a quaternionic vector space, \textit{but not the other eigenspaces}. 
\end{remark}

\begin{remark}\label{remark index}Note that in dimensions $m\equiv 1,2 \text{ }(\textup{mod } 8)$ we can use index theoretical informations to deduce that the dimension of the kernel of $\slashed{D}^{u_0}$ can not decrease along homotopies of $u_0$ if we have $\textup{dim}_\mathbb{K}\textup{ker}(\slashed{D}^{u_0})=1$. To be more precise, for $f\colon M\to N$ we have an index $\textup{ind}_{f^*TN}(M)\in KO^{-m}(pt)$, c.f. \cite[p. 151]{LaMi}. Using the isomorphism $KO^{-m}(pt)\cong \mathbb{Z}_2$ if $m\equiv 1,2 \text{ }(\textup{mod } 8)$ it holds that \cite[Theorem 7.13. on page 151]{LaMi}
	\begin{align*}
	\textup{ind}_{f^*TN}(M)=\begin{cases}
	[\textup{dim}_{\mathbb{C}}\textup{ker}(\slashed{D}^f)]_{\mathbb{Z}_2} &\text{if }m\equiv 1 \text{ }(\textup{mod } 8),\\
	[\textup{dim}_{\mathbb{H}}\textup{ker}(\slashed{D}^f)]_{\mathbb{Z}_2}  &\text{if }m\equiv 2 \text{ }(\textup{mod } 8).\\
	\end{cases}
	\end{align*}	
	 Since the index is invariant under homotopies \cite[Corollary 4.4.]{AmmGin} and $\textup{ind}_{u_0^*TN}(M)\neq 0$ we have that $\textup{ind}_{g^*TN}(M)\neq 0$ for any $g\colon M\to N$ homotopic to $u_0$. Hence,
	\[\textup{dim}_\mathbb{K}\textup{ker}(\slashed{D}^{g})\ge 1=\textup{dim}_\mathbb{K}\textup{ker}(\slashed{D}^{u_0}).\]
\end{remark}

\begin{lemma}[Uniform bounds for the resolvents]\label{lemma resolvents uniform bounds}Assume we are in the situation of Lemma \ref{lemma kern lokal konstant}. We consider the resolvent $R(\mu,\slashed{D}^{\pi\circ u_t})\colon \Gamma_{L^2}\to\Gamma_{L^2}$ of $\slashed{D}^{\pi\circ u_t}\colon \Gamma_{W^1_2}\to \Gamma_{L^2}$. By Lemma \ref{lemma ellipticregnonsmooth} we know that the restriction 
	\[R(\mu,\slashed{D}^{\pi\circ u_t})\colon \Gamma_{L^p}\to\Gamma_{W^1_p}\]
	is well-defined and bounded for any $2\le p<\infty.$ 
	If $R>0$ is small enough, then there exists $C=C(p,R)>0$ s.t.
	\[\sup_{|\mu|=\frac{\Lambda}{2}}\|R(\mu, \slashed{D}^{\pi\circ u_t})\|_{L(L^p,W^1_p)}<C\]
	for all $u\in B^T_R(v_0)\cap \{u|_{t=0}=u_0\}$, $t\in[0,T]$.
\end{lemma}
\begin{proof}First we uniformly bound
	\[\sup_{|\mu|=\frac{\Lambda}{2}}\|R(\mu,P^{u_t,u_0}\slashed{D}^{\pi\circ u_t}(P^{u_t,u_0})^{-1})\|_{L(L^p,W^1_p)}\]
	in terms of the resolvent of $\slashed{D}^{u_0}$. To that end, let $\psi\in \Gamma_{C^1}(\Sigma M \otimes u_0^*TN)$ be arbitrary. Using Lemma \ref{lemma Kruemmungsargument} and \eqref{eq z11} we have
	\begin{align*}
	&\|P^{u_t,u_0}\slashed{D}^{\pi\circ u_t}(P^{u_t,u_0})^{-1}\psi-\slashed{D}^{u_0}\psi\|_{L^p}\le C(R)2R\|\psi\|_{W^1_p}
	\end{align*}
	Choosing any $\theta\in (0,1)$ we thus have
	\begin{align*}
	\|P^{u_t,u_0}\slashed{D}^{\pi\circ u_t}(P^{u_t,u_0})^{-1}-\slashed{D}^{u_0}\|_{L(W^1_p,L^p)}\le \theta \min_{|\mu|=\frac{\Lambda}{2}}\frac{1}{\|R(\mu,\slashed{D}^{u_0})\|_{L(L^p,W^1_p)}}
	\end{align*}
	for all $u,v\in B_R^T(v_0)\cap\{u|_{t=0}=u_0\}$, $t\in[0,T]$, $R>0$ small enough. It is a standard fact from functional analysis that this implies
	\begin{align*}
	\|R(\mu,P^{u_t,u_0}\slashed{D}^{\pi\circ u_t}(P^{u_t,u_0})^{-1}\|_{L(L^p,W^1_p)}\le \frac{1}{1-\theta}\|R(\mu,\slashed{D}^{u_0})\|_{L(L^p,W^1_p)}
	\end{align*}
	for all $u,v\in B_R^T(v_0)\cap\{u|_{t=0}=u_0\}$, $t\in[0,T]$, $|\mu|=\frac{\Lambda}{2}$, $R>0$ small enough. The lemma now follows from the uniform bounds for $\|(P^{u_t,u_0})^{-1}\|_{L(W^1_p,W^1_p)}$ obtained in Corollary \ref{cor parallel bdd w1p}.
\end{proof}

In the following, we will construct a particular solution of the constraint equation \eqref{dhmhf2} with the strategy outlined in the introduction. For this solution we will show the estimates which are necessary for the contraction argument.

\begin{lemma}\label{lemma Anteil im Kern nicht null}In the situation of Lemma \ref{lemma kern lokal konstant} let $\psi_0\in\textup{ker}(\slashed{D}^{u_0})$ with $\|\psi_0\|_{L^2}=1$.  We define
	\[\sigma(u_t):=P^{u_0,u_t}\psi_0.\]
	Using the decomposition $\Gamma_{L^2}=\textup{ker}(\slashed{D}^{\pi\circ u_t})\oplus(\textup{ker}(\slashed{D}^{\pi\circ u_t}))^\bot$ we write
	\[\sigma(u_t)=\sigma_1(u_t)+\sigma_2(u_t).\]
	Then it holds that
	\begin{align}\label{Gleichung z4}
	\sqrt{\frac12}\le \|\sigma_1(u_t)\|_{L^2(M)}\le 1
	\end{align}
	for all $u\in B^T_R(v_0)\cap \{u|_{t=0}=u_0\}$, $t\in[0,T]$. In particular $\sigma_1(u_t)\neq 0$.
\end{lemma}
\begin{proof} We write $\sigma:= \sigma(u_t)$ and $\sigma_i:=\sigma_i(u_t)$. By \eqref{eq 111} we have
	\begin{align*}
	\|\slashed{D}^{\pi\circ u_t}(\sigma_1+\sigma_2)\|_{L^2}^2&\le C(R)R\|\psi_0\|_{L^2}^2<\frac12\Lambda \|\psi_0\|_{L^2}^2=\frac12\Lambda (\|\sigma_1\|_{L^2}^2+\|\sigma_2\|_{L^2}^2)
	\end{align*}
	(the second inequality is just due to our choice of $R$ in the proof of Lemma \ref{lemma kern lokal konstant}).	Moreover the Min-Max principle yields
	\begin{align*}
	\|\slashed{D}^{\pi\circ u_t}(\sigma_1+\sigma_2)\|_{L^2}^2=\|\slashed{D}^{\pi\circ u_t}\sigma_2\|_{L^2}^2\ge \Lambda\|\sigma_2\|_{L^2}^2.
	\end{align*}
	Combining these two inequalities, we get
	\[\Lambda\|\sigma_2\|_{L^2}^2<\frac12\Lambda(\|\sigma_1\|_{L^2}^2+\|\sigma_2\|_{L^2}^2),\]
	hence
	\[\|\sigma_2\|_{L^2}^2<\|\sigma_1\|_{L^2}^2\]
	and the lemma follows.
\end{proof}

Let us assume that we are in the situation of Lemma \ref{lemma Anteil im Kern nicht null}. Let $\Lambda>0$ be as in Lemma \ref{lemma kern lokal konstant}. Let $\gamma\colon [0,2\pi]\rightarrow \mathbb{C}$ be defined by $\gamma(x):=\frac{\Lambda}{2}e^{i x }$. Then for every $u\in B_R^T(v_0)\cap\{u|_{t=0}=u_0\}$ the mapping
\begin{align*}
\Gamma_{L^2}(\Sigma M\otimes (\pi\circ u_t)^*TN)&\to \Gamma_{L^2}(\Sigma M\otimes (\pi\circ u_t)^*TN),\\
s&\mapsto -\frac{1}{2\pi i}\int_{\gamma}R(\mu,\slashed{D}^{\pi\circ u_t})s\,d\mu,
\end{align*}
where $R(\mu,\slashed{D}^{\pi\circ u_t})\colon \Gamma_{L^2}\to \Gamma_{L^2}$ is the resolvent of $\slashed{D}^{\pi\circ u_t}\colon \Gamma_{W^1_2}\to \Gamma_{L^2}$, is the orthogonal projection onto $\textup{ker}(\slashed{D}^{\pi\circ u_t})$, c.f. \cite[Theorem 6.17 on p. 178]{Kato} or \cite[Theorem A.4.5. and Corollary A.4.6. i), ii), iii)]{Now}.

\begin{lemma}\label{lemma zentrale Spinorabschaetzung}In the situation of Lemma \ref{lemma Anteil im Kern nicht null} we define
	\[\tilde{\psi}(u_t):=-\frac{1}{2\pi i}\int_{\gamma}R(\mu,\slashed{D}^{\pi\circ u_t})\sigma(u_t)\,d\mu\]
	for every $u\in B_R^T(v_0)\cap\{u|_{t=0}=u_0\}$. In particular $\tilde{\psi}(u_t)\in\textup{ker}(\slashed{D}^{\pi\circ u_t})\subset\Gamma_{C^0}(\Sigma M\otimes (\pi\circ u_t)^*TN)$ and $\tilde{\psi}(u_t)\neq 0$. We write
	\[\psi(u_t):=\frac{\tilde{\psi}(u_t)}{\|\tilde{\psi}(u_t)\|_{L^2(M)}}.\]
	Let $\psi^A(u_t)$, $A=1,\ldots,q$, be the uniquely determined (global) sections of $\Sigma M$ s.t.
	\[\psi(u_t)=\psi^A(u_t)\otimes(\partial_A\circ \pi\circ u_t).\]
	If $\varepsilon >0$ and $T>0$ are small enough, then there exists $C=C(R,\varepsilon,\psi_0)>0$ s.t.
	\begin{align}\label{eq z22}
	\|P^{u_t,v_s}\tilde{\psi}(u_t)(x)-\tilde{\psi}(v_s)(x)\|\le C\|u_t-v_s\|_{C^0(M,\mathbb{R}^q)},
	\end{align}
	and
	\begin{align}\label{eq zz22}
	\|\psi^A(u_t)(x)-\psi^A(v_s)(x)\|_{\Sigma_xM}\le C \|u_t-v_s\|_{C^0(M,\mathbb{R}^q)}
	\end{align}
	for all $u,v\in B_R^T(v_0)\cap \{u|_{t=0}=u_0\}$, $A=1,\ldots,q$, $x\in M$, $s,t\in[0,T]$.
\end{lemma}
\begin{proof} We define $\tilde{\psi}^A(u_t)\in\Gamma(\Sigma M)$ by
	\[\tilde{\psi}(u_t)=\tilde{\psi}^A(u_t)\otimes(\partial_A\circ \pi\circ u_t).\]
	We will prove the lemma in three steps. First we show \eqref{eq z22}. Then we use \eqref{eq z22} to get
	\begin{align}\label{eq z23}
	\|\tilde{\psi}^A(u_t)(x)-\tilde{\psi}^A(v_s)(x)\|_{\Sigma_xM}\le C(R,\varepsilon,\psi_0) \|u_t-v_s\|_{C^0(M,\mathbb{R}^q)}
	\end{align}
	for all $u,v\in B_R^T(v_0)\cap \{u|_{t=0}=u_0\}$, $A=1,\ldots,q$, $x\in M$, $s,t\in[0,T]$.
	From \eqref{eq z22} and \eqref{eq z23} the equation \eqref{eq zz22} will follow from a short computation.
	
	\textbf{Step 1: Proof of \eqref{eq z22}:} In the following we use the well-known resolvent identity  
	\[R(\lambda, T)-R(\lambda,T_0)=R(\lambda,T)\circ (T-T_0)\circ R(\lambda,T_0),\]
	where $T,T_0\colon D(T)\subset X\to X$ are two operators on a Banach space $X$ with the same domain of definition and $\lambda$ is in the intersection of their resolvent sets. We calculate
	\begin{align*}
	&P^{u_t,v_s}\tilde{\psi}(u_t)-\tilde{\psi}(v_s)\\
	&=-\frac{1}{2\pi i}\Big(\int_{\gamma}P^{u_t,v_s}R(\mu,\slashed{D}^{\pi\circ u_t})(P^{u_t,v_s}\big)^{-1}P^{u_t,v_s}P^{u_0,u_t}\psi_0\,d\mu\\
	&\hphantom{=-\frac{1}{2\pi i}\Big(}-\int_{\gamma}R(\mu,\slashed{D}^{\pi\circ v_s})P^{u_0,v_s}\psi_0\,d\mu\Big)\\
	&=-\frac{1}{2\pi i}\Big(\int_{\gamma}R(\mu,P^{u_t,v_s}\slashed{D}^{\pi\circ u_t}(P^{u_t,v_s}\big)^{-1})P^{u_t,v_s}P^{u_0,u_t}\psi_0\,d\mu\\
	&\hphantom{=-\frac{1}{2\pi i}\Big(}-\int_{\gamma}R(\mu,\slashed{D}^{\pi\circ v_s})P^{u_0,v_s}\psi_0\,d\mu\Big)\\
	&=-\frac{1}{2\pi i}\int_{\gamma}R(\mu,P^{u_t,v_s}\slashed{D}^{\pi\circ u_t}(P^{u_t,v_s}\big)^{-1})\bigg(P^{u_t,v_s}P^{u_0,u_t}\psi_0-P^{u_0,v_s}\psi_0\bigg)\,d\mu\\
	&\hphantom{=}-\frac{1}{2\pi i}\int_{\gamma}\bigg(R(\mu,P^{u_t,v_s}\slashed{D}^{\pi\circ u_t}(P^{u_t,v_s}\big)^{-1})-R(\mu,\slashed{D}^{\pi\circ v_s})\bigg)P^{u_0,v_s}\psi_0\,d\mu\\
	&=-\frac{1}{2\pi i}\int_{\gamma}R(\mu,P^{u_t,v_s}\slashed{D}^{\pi\circ u_t}(P^{u_t,v_s}\big)^{-1})\bigg(P^{u_t,v_s}P^{u_0,u_t}\psi_0-P^{u_0,v_s}\psi_0\bigg)\,d\mu\\
	&\hphantom{=}-\frac{1}{2\pi i}\int_{\gamma}\bigg(R(\mu,P^{u_t,v_s}\slashed{D}^{\pi\circ u_t}(P^{u_t,v_s}\big)^{-1}) \circ\Big(P^{u_t,v_s}\slashed{D}^{\pi\circ u_t}(P^{u_t,v_s}\big)^{-1}-\slashed{D}^{\pi\circ v_s}\Big)\circ\\
	&\hphantom{\hphantom{=}-\frac{1}{2\pi i}\int_{\gamma}\bigg(}R(\mu,\slashed{D}^{\pi\circ v_s})\bigg)P^{u_0,v_s}\psi_0\,d\mu.
	\end{align*}
	
	Therefore we get for $p$ large enough
	\begin{align*}
	&\|P^{u_t,v_s}\tilde{\psi}(u_t)(x)-\tilde{\psi}(v_s)(x)\|\\
	&= \|P^{v_s,u_0}P^{u_t,v_s}\tilde{\psi}(u_t)-P^{v_s,u_0}\tilde{\psi}(v_s)\|_{C^0(M)}\\
	&\le C(u_0)\|P^{v_s,u_0}P^{u_t,v_s}\tilde{\psi}(u_t)-P^{v_s,u_0}\tilde{\psi}(v_s)\|_{W^1_p(M)}\\
	&\le C_1\|P^{u_t,v_s}\tilde{\psi}(u_t)-\tilde{\psi}(v_s)\|_{W^1_p(M)}\\
	&\le C_2\int_{\gamma}\bigg\|R(\mu,P^{u_t,v_s}\slashed{D}^{\pi\circ u_t}(P^{u_t,v_s}\big)^{-1})\bigg(P^{u_t,v_s}P^{u_0,u_t}\psi_0-P^{u_0,v_s}\psi_0\bigg)\bigg\|_{W^1_p(M)}\,d\mu\\
	&+C_2\int_{\gamma}\bigg\|\bigg(R(\mu,P^{u_t,v_s}\slashed{D}^{\pi\circ u_t}(P^{u_t,v_s}\big)^{-1}) \circ\Big(P^{u_t,v_s}\slashed{D}^{\pi\circ u_t}(P^{u_t,v_s}\big)^{-1}-\slashed{D}^{\pi\circ v_s}\Big)\circ\\
	&\hphantom{\hphantom{=}-\frac{1}{2\pi i}\int_{\gamma}\bigg(}R(\mu,\slashed{D}^{\pi\circ v_s})\bigg)P^{u_0,v_s}\psi_0\bigg\|_{W^1_p(M)}\,d\mu\\
	&\le C_3 \sup_{\mu\in\textup{Im}(\gamma)}\|R(\mu,P^{u_t,v_s}\slashed{D}^{\pi\circ u_t}(P^{u_t,v_s}\big)^{-1})\|_{L(L^p,W^1_p)}\|P^{u_t,v_s}P^{u_0,u_t}\psi_0-P^{u_0,v_s}\psi_0\|_{L^p}\\
	&\hphantom{\le} + C_3\sup_{\mu\in\textup{Im}(\gamma)}\|R(\mu,P^{u_t,v_s}\slashed{D}^{\pi\circ u_t}(P^{u_t,v_s}\big)^{-1})\|_{L(L^p,W^1_p)}\sup_{\mu\in\textup{Im}(\gamma)}\|R(\mu,\slashed{D}^{\pi\circ v_s})\|_{L(L^p,W^1_p)}\\
	&\hphantom{+C_3}\|P^{u_t,v_s}\slashed{D}^{\pi\circ u_t}(P^{u_t,v_s}\big)^{-1}-\slashed{D}^{\pi\circ v_s}\|_{L(W^1_p,L^p)}\|P^{u_0,v_s}\psi_0\|_{L^p}
	\end{align*}
	
	The norms of the resolvents are uniformly bounded by Lemma \ref{lemma resolvents uniform bounds} and Corollary \ref{cor parallel bdd w1p}. Moreover, Lemma \ref{lemma Kruemmungsargument} yields
	\begin{align*}
	\|P^{u_t,v_s}\slashed{D}^{\pi\circ u_t}(P^{u_t,v_s}\big)^{-1}-\slashed{D}^{\pi\circ v_s}\|_{L(W^1_p,L^p)}\ \le C(R) \|u_t-v_s\|_{C^0(M,\mathbb{R}^q)}.
	\end{align*}
	Using Lemma \ref{lemma parallel transport deviation} we get
	\[\|P^{u_t,v_s}P^{u_0,u_t}\psi_0-P^{u_0,v_s}\psi_0\|_{L^p}\le C(\varepsilon,\psi_0)\|u_t-v_s\|_{C^0(M,\mathbb{R}^q)}.\]
	Putting everything together we have shown \eqref{eq z22}.

	\textbf{Step 2: Proof of \eqref{eq z23}:} We have
	\begin{align*}
	&\|\tilde{\psi}^A(u_t)|_x-\tilde{\psi}^A(v_s)|_x\|_{\Sigma_xM}\\
	&\le\|\tilde{\psi}(u_t)|_x-\tilde{\psi}(v_s)|_x\|_{\Sigma_xM\otimes\mathbb{R}^q}\\
	&\le \|P^{u_t,v_s}\tilde{\psi}(u_t)|_x-\tilde{\psi}(v_s)|_x\|_{\Sigma_xM\otimes\mathbb{R}^q} + \|P^{u_t,v_s}\tilde{\psi}(u_t)|_x-\tilde{\psi}(u_t)|_x\|_{\Sigma_xM\otimes\mathbb{R}^q}\\
	\end{align*}
	In the following we estimate the two summands separately. Using the fact that the differential of $i\colon N\rightarrow\mathbb{R}^q$ is an isometry and \eqref{eq z22} we get
	\begin{align*}
	\|P^{u_t,v_s}\tilde{\psi}(u_t)|_x-\tilde{\psi}(v_s)|_x\|_{\Sigma_xM\otimes\mathbb{R}^q}&=\|P^{u_t,v_s}\tilde{\psi}(u_t)|_x-\tilde{\psi}(v_s)|_x\|_{\Sigma_xM\otimes T_{(\pi\circ v_s(x))}N}\\
	&\le C(R,\varepsilon, \psi_0)\|u_t-v_s\|_{C^0(M,\mathbb{R}^q)}.
	\end{align*}
	It remains to find an appropriate estimate for $\|P^{u_t,v_s}\tilde{\psi}(u_t)|_x-\tilde{\psi}(u_t)|_x\|_{\Sigma_xM\otimes\mathbb{R}^q}$. To that end, let $\gamma(h):=\textup{exp}_{(\pi\circ u_t)(x)}(h\textup{exp}^{-1}_{(\pi\circ u_t)(x)}(\pi\circ v_s(x))$, $h\in[0,1]$, be the unique shortest geodesic of $N$ from $(\pi\circ u_t)(x)$ to $(\pi\circ v_s)(x)$. Let $X\in T_{\gamma(0)}N$ be given and denote by $X(h)$ the unique parallel vector field (of $N$) along $\gamma$ with $X(0)=X$. Then we have
	\begin{align*}
	P^{u_t,v_s}X-X=X(1)-X(0)=\int_{0}^{1}\frac{d}{dh}\bigg|_{h=\tau}X(h)\,d\tau=\int_{0}^{1}II(\gamma'(h),X(h))\, dh.
	\end{align*}
	Therefore
	\begin{align*}
	\|P^{u_t,v_s}X-X\|_{\mathbb{R}^q}\le C_1 \sup_{h\in[0,1]}\|\gamma'(h)\|_N \sup_{h\in[0,1]}\|X(h)\|_N=C_1\|\gamma'(0)\|_N\|X\|_N
	\end{align*}
	where $C_1$ only depends on the Riemannian manifold $N\subset\mathbb{R}^q$. Using \eqref{eq z4} and the fact that $\pi\colon \mathbb{R}^q\rightarrow\mathbb{R}^q$ is (globally) Lipschitz continuous we have that
	\begin{align*}
	\|\gamma'(0)\|_N\le d^N((\pi\circ u_t)(x),(\pi\circ v_s)(x))\le C_2 \|u_t(x)-v_s(x)\|_{\mathbb{R}^q}.
	\end{align*}
	Hence,
	\begin{align*}
	\|P^{u_t,v_s}X-X\|_{\mathbb{R}^q}\le C_3\|u_t(x)-v_s(x)\|_{\mathbb{R}^q}\|X\|_N.
	\end{align*}
	This implies
	\begin{align*}
	\|P^{u_t,v_s}\tilde{\psi}(u_t)|_x-\tilde{\psi}(u_t)|_x\|_{\Sigma_xM\otimes\mathbb{R}^q}\le C(R,\varepsilon,\psi_0)\|u_t(x)-v_s(x)\|_{\mathbb{R}^q},
	\end{align*}
	hence \eqref{eq z23} holds.
	
	\textbf{Step 3: Proof of \eqref{eq zz22}:} We have
	\[\psi^A(u_t)(x)=\frac{\tilde{\psi}^A(u_t)(x)}{\|\tilde{\psi}(u_t)\|_{L^2(M)}}.\] 
	Using \eqref{eq z22} and \eqref{eq z23} we get
	\begin{align*}
	&\|\psi^A(u_t)(x)-\psi^A(v_s)(x)\|\\
	&=\bigg\|\frac{\tilde{\psi}^A(u_t)(x)}{\|\tilde{\psi}(u_t)\|_{L^2(M)}}-\frac{\tilde{\psi}^A(u_t)(x)}{\|\tilde{\psi}(v_s)\|_{L^2(M)}}+\frac{\tilde{\psi}^A(u_t)(x)}{\|\tilde{\psi}(v_s)\|_{L^2(M)}}-\frac{\tilde{\psi}^A(v_s)(x)}{\|\tilde{\psi}(v_s)\|_{L^2(M)}}\bigg\|\\
	&\le\frac{\|\tilde{\psi}^A(u_t)(x)\|}{\|\tilde{\psi}(u_t)\|_{L^2(M)}\|\tilde{\psi}(v_s)\|_{L^2(M)}}\Big|\|\tilde{\psi}(v_s)\|_{L^2(M)}-\|\tilde{\psi}(u_t)\|_{L^2(M)}\Big|\\
	&\hphantom{\le}+ \frac{1}{\|\tilde{\psi}(v_s)\|_{L^2(M)}}\|\tilde{\psi}^A(u_t)(x)-\tilde{\psi}^A(v_s)(x)\|\\
	&=\frac{\|\tilde{\psi}^A(u_t)(x)\|}{\|\tilde{\psi}(u_t)\|_{L^2(M)}\|\tilde{\psi}(v_s)\|_{L^2(M)}}\Big|\|\tilde{\psi}(v_s)\|_{L^2(M)}-\|P^{u_t,v_s}\tilde{\psi}(u_t)\|_{L^2(M)}\Big|\\
	&\hphantom{\le}+ \frac{1}{\|\tilde{\psi}(v_s)\|_{L^2(M)}}\|\tilde{\psi}^A(u_t)(x)-\tilde{\psi}^A(v_s)(x)\|\\
	&\le\frac{\|\tilde{\psi}^A(u_t)(x)\|}{\|\tilde{\psi}(u_t)\|_{L^2(M)}\|\tilde{\psi}(v_s)\|_{L^2(M)}}\|P^{u_t,v_s}\tilde{\psi}(u_t)-\tilde{\psi}(v_s)\|_{L^2(M)}\\
	&\hphantom{\le}+ \frac{1}{\|\tilde{\psi}(v_s)\|_{L^2(M)}}\|\tilde{\psi}^A(u_t)(x)-\tilde{\psi}^A(v_s)(x)\|\\
	&\le\Big( \frac{\|\tilde{\psi}^A(u_t)(x)\|}{\|\tilde{\psi}(u_t)\|_{L^2(M)}\|\tilde{\psi}(v_s)\|_{L^2(M)}}+\frac{1}{\|\tilde{\psi}(v_s)\|_{L^2(M)}}\Big)C(R,\varepsilon,\psi_0) \|u_t-v_s\|_{C^0(M,\mathbb{R}^q)}.
	\end{align*}
	Moreover, the $L^2$-norms in the denominators are uniformly bounded by Lemma \ref{lemma Anteil im Kern nicht null} and $\|\tilde{\psi}^A(u_t)(x)\|$ is uniformly bounded by \eqref{eq z23} and the triangle inequality.	This completes the proof of the lemma.
\end{proof}
\newpage

\section{Short time existence}
In this section we prove Theorem \ref{thereom short time existence}. As we already mentioned in the introduction, the proof is inspired by \cite{ChenJostSunZhu}. A contraction argument with a similar structure can be found in \cite[Proof of Theorem 5.2.1 on page 111]{LinWang}. For the latter we also recommend \cite{LinWangsupp} as a supplement.

\begin{proof}[Proof of Theorem \ref{thereom short time existence}]
	\textbf{Step 1: Solving the equation in $\mathbb{R}^q$:}
	In this step we want to find a solution $u\colon [0,T]\times M \rightarrow \mathbb{R}^q$, $\psi_t\colon M\rightarrow \Sigma M\otimes (\pi\circ u_t)^*TN$ of 
	\begin{align}
	\label{eq z6}\begin{cases}
	\partial_tu^A-\Delta u^A =F_1^A(u) +F_2^A(u,\psi) &\text{ on }(0,T)\times M,\text{ }A=1,\ldots,q,\\
	\slashed{D}^{\pi\circ u_t}\psi_t=0&\text{ on }[0,T]\times M,\\
	u([0,T]\times M)\subset N_\delta,\\
	u|_{t=0}=u_0,\\
	\psi|_{t=0}=\psi_0,\\
	\|\psi_t\|_{L^2(M)}=1 &\text{ on }[0,T],\\
	\textup{dim}_\mathbb{K}\textup{ker}(\slashed{D}^{\pi\circ u_t})=1&\text{ on }[0,T],
	\end{cases}
	\end{align}
	where $u_0\in C^{2+\alpha}(M,N)$ with $\textup{dim}_\mathbb{K}\textup{ker}(\slashed{D}^{u_0})=1$, and $\psi_0\in \textup{ker}(\slashed{D}^{u_0})$ with $\|\psi_0\|_{L^2(M)}=1$ are given.
	
	We choose $\varepsilon, \delta, R,$ and $T$ as in Table \ref{tabelle konstanten}. By making $\varepsilon, R,$ and $T$ smaller if necessary, Lemmas \ref{lemma kern lokal konstant} and \ref{lemma zentrale Spinorabschaetzung} hold. Recall that our choices imply in particular that $u([0,T]\times M)\subset N_\delta$ for all  $u\in B_R^T(v_0)\cap \{u|_{t=0}=u_0\}$. Let $\psi(u_t)$ and $\psi^A(u_t)$, $A=1,\ldots, q$ be as in Lemma \ref{lemma zentrale Spinorabschaetzung}. In particular we have 
	\begin{align*}
	\slashed{D}^{\pi\circ u_t}\psi(u_t)=0&\text{ on }[0,T]\times M,\\
	\|\psi(u_t)\|_{L^2(M)}=1&\text{ on }[0,T],\\
	\textup{dim}_\mathbb{K}\textup{ker}(\slashed{D}^{\pi\circ u_t})=1&\text{ on }[0,T],
	\end{align*}
	and $\psi(u_t)|_{t=0}=\psi(u_0)=\psi_0$ for all $u\in B_R^T(v_0)\cap \{u|_{t=0}=u_0\}$, $t\in [0,T]$.
	
	Plugging $\psi(u_t)$ into the first line of \eqref{eq z6} it remains to find $u\in B_R^T(v_0)\cap \{u|_{t=0}=u_0\}$ that solves
	\begin{align}\label{eq z10}
	\partial_tu^A-\Delta u^A =F_1^A(u) +F_2^A(u,\psi(u)) &\text{ on }[0,T]\times M,\text{ }A=1,\ldots,q.
	\end{align}
	To that end, for $u\in B_R^T(v_0)\cap \{u|_{t=0}=u_0\}$ we consider
	\begin{align*}
	(Lu)(t,x)=v_0(t,x)+\int_0^t\int_M p(x,y,t-\tau)\left(F_1(u_\tau)(y) +F_2(u_\tau,\psi(u_\tau))(y)\right) \,dV(y)d\tau
	\end{align*}
	as in Section \ref{section fixed point operator}.
	
	In the following we show that if $T$ is small enough, then it holds that
	\begin{enumerate}
		\item $L(B_R^T(v_0)\cap \{u|_{t=0}=u_0\})\subset B_R^T(v_0)\cap \{u|_{t=0}=u_0\}$,
		\item $\|Lu-Lv\|_{X_T}\le \frac12 \|u-v\|_{X_T}$ for all $u,v\in B_R^T(v_0)\cap \{u|_{t=0}=u_0\}$.
	\end{enumerate}
	We start with i): Let $u\in B_R^T(v_0)\cap \{u|_{t=0}=u_0\}$ and consider
	\[(Lu-v_0)^A(t,x)=\int_0^t\int_M p(x,y,t-\tau)\left(F_1^A(u_\tau)(y) +F_2^A(u_\tau,\psi(u_\tau))(y)\right) \,dV(y)d\tau.\]
	We have
	\begin{align*}
	|(Lu-v_0)^A(t,x)|\le t \sup_{(s,z)\in [0,T]\times M}|F_1^A(u_s)(z) +F_2^A(u_s,\psi(u_s))(z)|
	\end{align*}
	and 
	\begin{align*}
	|\nabla_x (Lu-v_0)^A(t,x)|\le C \sqrt{t}\sup_{(s,z)\in [0,T]\times M}|F_1^A(u_s)(z) +F_2^A(u_s,\psi(u_s))(z)|
	\end{align*}
	for all $(t,x)\in [0,T]\times M$, $A=1,\ldots,q$, provided that $T\le 1$ is small enough.\footnote{Here we use that $p\ge 0$ and $\int_M p(x,y,t)\,dV(y)=1$ for all $(x,t)\in M\times (0,\infty)$. Moreover, we use that there exists $C>0$ s.t. $\int_{0}^{t}\int_M|\nabla_x p(x,y,s)|\,dV(y)ds\le C \sqrt{t}$ for all $(x,t)\in M\times [0,1]$. The latter is not difficult to show. It follows directly from the construction of the heat kernel (see e.g. \cite{Chavel}). It is shown in detail in \cite{LinWangsupp} or \cite{JWDissertation}.} Since $u\in B_R^T(v_0)$ we have 
	\begin{align}\label{eq z14}
	\|u\|_{X_T}\le \|u-v_0\|_{X_T}+\|v_0\|_{X_T} \le R+ \|v_0\|_{X_T}
	\end{align}
	hence
	\begin{align*}
	\sup_{(s,z)\in [0,T]\times M}|F_1^A(u_s)(z)|\le C(R,\|v_0\|_{X_1})
	\end{align*}
	(recall that $\pi\colon\mathbb{R}^q\rightarrow\mathbb{R}^q$ has compact support). By \eqref{eq zz22} and the triangle inequality we have
	\begin{align}\label{eq z16}
	\|\psi^A(u_s)(z)\|\le C_1(R,\psi_0)
	\end{align}
	(recall that our choice of constants in Table \ref{tabelle konstanten} implies in particular that \eqref{eq z11} holds).
	Therefore 
	\begin{align*}
	\sup_{(s,z)\in [0,T]\times M}|F_2^A(u_s,\psi(u_s))(z)|\le C_2(R,\psi_0).
	\end{align*}
	We have shown that if $T>0$ is small enough, then
	\begin{align*}
	|(Lu-v_0)^A(t,x)|&\le C_3(R,\psi_0) t,\\
	|\nabla_x (Lu-v_0)^A(t,x)|&\le C_3(R,\psi_0) \sqrt{t}
	\end{align*}
	for all $(t,x)\in [0,T]\times M$, $A=1,\ldots,q$, and for all $u\in B_R^T(v_0)\cap \{u|_{t=0}=u_0\}$. Hence for $T>0$ small enough we have $Lu\in B_R^T(v_0)$. This implies i), since $Lu|_{t=0}=u_0$.
	
	Next we show ii): Let $u,v\in B_R^T(v_0)\cap \{u|_{t=0}=u_0\}$. We have
	\begin{align*}
	&(Lu-Lv)^A(t,x)\\
	&=\int_0^t\int_M p(x,y,t-\tau)\left(F_1^A(u_\tau)(y) -F_1^A(v_\tau)(y)\right) \,dV(y)d\tau\\
	&+\int_0^t\int_M p(x,y,t-\tau)\left(F_2^A(u_\tau,\psi(u_\tau))(y)-F_2^A(v_\tau,\psi(v_\tau))(y)\right) \,dV(y)d\tau
	\end{align*}
	As above we get
	\begin{align}\label{eq z12}\begin{split}
	&|(Lu-Lv)^A(t,x)|\\
	&\le t \sup_{(s,z)\in [0,T]\times M}|F_1^A(u_s)(z) - F_1^A(v_s)(z)|\\
	&+t\sup_{(s,z)\in [0,T]\times M}|F_2^A(u_s,\psi(u_s))(z)-F_2^A(v_s,\psi(v_s))(z)|\end{split}
	\end{align}
	and
	\begin{align}\label{eq z13}\begin{split}
	&|\nabla_x(Lu-Lv)^A(t,x)|\\
	&\le C \sqrt{t} \sup_{(s,z)\in [0,T]\times M}|F_1^A(u_s)(z) - F_1^A(v_s)(z)|\\
	&+C\sqrt{t}\sup_{(s,z)\in [0,T]\times M}|F_2^A(u_s,\psi(u_s))(z)-F_2^A(v_s,\psi(v_s))(z)|\end{split}
	\end{align}
	for all $(t,x)\in [0,T]\times M$ provided that $T\le 1$ is small enough.
	We calculate
	\begin{align*}
	&|F_1^A(u) - F_1^A(v)|\\
	&=|\pi^A_{BC}(u)\langle \nabla u^B,\nabla u^C\rangle-\pi^A_{BC}(v)\langle \nabla v^B,\nabla v^C\rangle|\\
	&=|\pi^A_{BC}(u)\big(\langle \nabla u^B, \nabla u^C\rangle - \langle \nabla v^B,\nabla v^C\rangle\big)\\
	&\hphantom{=}+\big(\pi^A_{BC}(u)-\pi^A_{BC}(v)\big) \langle \nabla v^B,\nabla v^C\rangle|\\
	&\le |\pi^A_{BC}(u)||\Big(\|\nabla u^B\|\|\nabla u^C-\nabla v^C\| + \|\nabla v^C\| \|\nabla u^B-\nabla v^B\|\Big)|+\\
	&\hphantom{\le}+|\pi^A_{BC}(u)-\pi^A_{BC}(v)| \|\nabla v^B\|\|\nabla v^C\|
	\end{align*}
	Using the fact that $\pi^A_{BC}\colon\mathbb{R}^q\to \mathbb{R}$ has compact support and is (globally) Lipschitz continuous together with \eqref{eq z14} we deduce
	\begin{align}\label{eq z15}
	\sup_{(s,z)\in [0,T]\times M}|F_1^A(u_s)(z) - F_1^A(v_s)(z)|\le C(R)\|u-v\|_{X_T}.
	\end{align}
	Using \eqref{eq z14},\eqref{eq z16}, and \eqref{eq zz22} an analogous calculation yields
	\begin{align}\label{eq z17}
	\sup_{(s,z)\in [0,T]\times M}|F_2^A(u_s,\psi(u_s))(z)-F_2^A(v_s,\psi(v_s))(z)|\le C(R,\psi_0)\|u-v\|_{X_T}.
	\end{align}
	Plugging \eqref{eq z15} and \eqref{eq z17} into \eqref{eq z12} and \eqref{eq z13} yields
	\begin{align*}
	|(Lu-Lv)^A(t,x)|&\le t C(R,\psi_0)\|u-v\|_{X_T},\\
	|\nabla_x(Lu-Lv)^A(t,x)|&\le \sqrt{t} C(R,\psi_0)\|u-v\|_{X_T}
	\end{align*}
	for all $(t,x)\in [0,T]\times M$, $A=1,\ldots,q$, and for all $u,v\in B_R^T(v_0)\cap \{u|_{t=0}=u_0\}$. Now ii) follows by choosing $T$ small enough.\\
	
	Applying the Banach fixed-point theorem we get a unique $u\in B_R^T(v_0)\cap\{u|_{t=0}=u_0\}$ with $Lu=u$.\\\newline	
	\textbf{Step 2: Regularity of the fixed point:}
	In this step we show that the fixed point $u$ is an element of $C^{1,2,\alpha}((0,T)\times M,\mathbb{R}^q)$. Equation \eqref{eq z16} implies that $F_1^A(u)$ and $F_2^A(u,\psi(u))$ are bounded on $[0,T]\times M$. Therefore the $L^p$-regularity for the heat equation yields
	\[u\in W^{1,2,p}((0,T)\times M)\]
	for all $p\in (1,\infty)$. Hence we have\footnote{To show that $W^{1,2,p}((0,T)\times M)\subset C^{0,1,\alpha}((0,T)\times M)$ for $p$ large enough (the spaces are defined as in \cite{BDPhil}) one needs the Sobolev embedding and interpolation theory.}
	\[u\in C^{0,1,\alpha}((0,T)\times M).\]
	This implies
	$\psi^A(u)\in\Gamma_{C^{0,0,\alpha}}(\Sigma M\to (0,T)\times M)$, i.e.,
	\begin{align}\label{eq z28}
	\sup_{x\in M}\|\psi^A(u_.)(x)\|_{C^{\frac\alpha 2}((0,T)\to \Sigma_xM)}<\infty,
	\end{align}
	and
	\begin{align}\label{eq z29}
	\sup_{t\in(0,T)}\|\psi^A(u_t)(.)\|_{\Gamma_{C^\alpha}(\Sigma M \to M)}<\infty.
	\end{align}
	Note that by Lemma \ref{lemma zentrale Spinorabschaetzung} we have
	\begin{align*}
	\|\psi^A(u_t)(x)-\psi^A(u_s)(x)\|&\le C \|u_t-u_s\|_{C^0(M,\mathbb{R}^q)}\\
	&\le C \|u\|_{C^{0,0,\alpha}((0,T)\times M)}|t-s|^{\frac\alpha2},
	\end{align*}
	hence we get \eqref{eq z28}. One can show \eqref{eq z29} with the techniques that we developed so far, details can be found in \cite{JWDissertation}. Since $\psi^A(u)\in\Gamma_{C^{0,0,\alpha}}(\Sigma M\to (0,T)\times M)$ and $u\in C^{0,1,\alpha}((0,T)\times M)$, we get
	\[F_1(u), F_2(u,\psi(u))\in C^{0,0,\alpha}((0,T)\times M).\]
	By the Hölder-regularity for the heat equation we deduce
	\[u\in C^{1,2,\alpha}((0,T)\times M,\mathbb{R}^q).\]
	\newline	
	\textbf{Step 3: The fixed point takes values in $N$:} First let $f\colon (0,T)\times M\to \mathbb{R}^q$ be an arbitrary function s.t. $f(t,.)\in C^2(M,\mathbb{R}^q)$ for all $t\in(0,T)$ and $f(.,p)\in C^1((0,T),\mathbb{R}^q)$ for all $p\in M$. In the following we write $\|.\|_2$ and $\langle.,.\rangle_2$ for the Euclidean norm and scalar product, respectively. Similarly we write $\|.\|_g$ and $\langle.,.\rangle_g$ for the norm and scalar product of the Riemannian manifold $(M,g)$, respectively. We define
	\[\rho\colon\mathbb{R}^q\to\mathbb{R}^q\]
	by $\rho(z):=z-\pi(z)$ and
	\[\varphi\colon (0,T)\times M \rightarrow \mathbb{R}\]
	by $\varphi(t,x):=\|\rho(f(t,x))\|_2^2=\sum_{A=1}^{q}|\rho^A(f(t,x))|^2$. A straight forward calculation yields (for details we refer to \cite{JWDissertation})
	\begin{align*}
	\big(\frac{\partial}{\partial t}-\Delta_x\big)\varphi(t,x)&=-2\sum_{A=1}^q \|\nabla_x(\rho^A\circ f)(t,x)\|_g^2\\
	&\hphantom{=}+2\Big\langle \rho(f(t,x)),\rho^A_B(f(t,x))\big(\frac{\partial}{\partial t}-\Delta_x\big)f^B(t,x)\Big\rangle_2\\
	&\hphantom{=}+2\Big\langle \rho(f(t,x)),\pi^A_{CB}(f(t,x))\langle \nabla_x f^C(t,x),\nabla_x f^B(t,x)\rangle_g\Big\rangle_2
	\end{align*}
	where $\rho^A_B(z):=\frac{\partial\rho^A}{\partial z_B}(z)$. Now let $f=u$ be the solution constructed in the first step. Then we have
	\begin{align*}
	\big(\frac{\partial}{\partial t}-\Delta_x\big)\varphi&=-2\sum_{A=1}^q \|\nabla_x(\rho^A\circ u)\|_g^2\\
	&\hphantom{=}+2\Big\langle \rho(u),\rho^A_B(u)\big(F_1^B(u)+F_2^B(u,\psi(u))\big)\Big\rangle_2\\
	&\hphantom{=}+2\Big\langle \rho(u),-F_1(u)\Big\rangle_2\\
	&=-2\sum_{A=1}^q \|\nabla_x(\rho^A\circ u)\|_g^2\\
	&\hphantom{=}+2\Big\langle \rho(u),-\pi^A_B(u)F_1^B(u)+\rho^A_B(u)F_2^B(u,\psi(u))\Big\rangle_2\\
	&=-2\sum_{A=1}^q \|\nabla_x(\rho^A\circ u)\|_g^2\\
	&\le 0.
	\end{align*}
	Here we used that $\Big\langle \rho(u),-\pi^A_B(u)F_1^B(u)+\rho^A_B(u)F_2^B(u,\psi(u))\Big\rangle_2=0$. This holds because of the following: let $(t,x)$ be arbitrary. Since $u(t,x)\in N_\delta$, we have that $\rho(u(t,x))=u(t,x)-\pi(u(t,x))\in T_{\pi(u(t,x))}^\bot N$. Moreover, $\big(\pi^A_B(u(t,x))F_1^B(u)(t,x)\big)_A\in T_{\pi(u(t,x))}N$ since 
	\[\pi^A_B(u(t,x))F_1^B(u)(t,x)=(d\pi^A)_{u(t,x)}(F_1(u)(t,x))=\big((d\pi)_{u(t,x)}(F_1(u)(t,x)\big)^A\]
	and $(d\pi)_{u(t,x)}\colon\mathbb{R}^q\to T_{\pi(u(t,x))}N$. Hence,
	\[\Big\langle \rho(u(t,x)),-\pi^A_B(u(t,x))F_1^B(u)(t,x)\Big\rangle_2=0.\]
	To see that $\Big\langle \rho(u(t,x)),\rho^A_B(u(t,x))F_2^B(u,\psi(u))(t,x)\Big\rangle_2=0$ we write
	\begin{align*}
	&\Big\langle \rho(u(t,x)),\rho^A_B(u(t,x))F_2^B(u,\psi(u))(t,x)\Big\rangle_2\\
	&=\Big\langle \rho(u(t,x)),F_2(u,\psi(u))(t,x)-\pi^A_B(u(t,x))F_2^B(u,\psi(u))(t,x)\Big\rangle_2		
	\end{align*}
	and note that by definition of $F_2$ we have that $F_2(u,\psi(u))(t,x)\in T_{\pi(u(t,x))}N$ and as above we have \[\big(\pi^A_B(u(t,x))F_2^B(u,\psi(u))(t,x)\big)_A=(d\pi)_{u(t,x)}\big(F_2(u,\psi(u))(t,x)\big)\in T_{\pi(u(t,x))}N.\]
	Since $\big(\frac{\partial}{\partial t}-\Delta_x\big)\varphi(t,x)\le 0$ for all $(t,x)$ and $\varphi(0,.)=0$ on $M$, the maximum principle for the heat equation yields $\varphi(t,x)\le 0$ for all $(t,x)$. The definition of $\varphi$ implies $\varphi\ge 0$, hence $\varphi(t,x)=0$ for all $(t,x)$. We have shown $u(t,x)\in N$ for all $(t,x)$.\\\newline
	\textbf{Step 4: Uniqueness of the solution:} Let $(u^1,\psi^1)$ and $(u^2,\psi^2)$ be two solutions of the heat flow for Dirac harmonic maps as in Theorem \ref{thereom short time existence}. In particular,  $u^i\colon [0,T]\times M\rightarrow N$, $\psi_t^i\colon M\rightarrow \Sigma M\otimes (u_t^i)^*TN$ solve \eqref{eq z6} and $u^i\in C^{1,2,\alpha}((0,T)\times M,N)$, $i=1,2$.\footnote{Note that here $T>0$ is just some $T$ s.t. Theorem \ref{thereom short time existence} holds. It does not need to be related to the $T$ we constructed in the first step.} Let $R>0$ be as in the first step (i.e., $L$ is a contraction on $B_R^{\hat{T}}(v_0)\cap\{u|_{t=0}=u_0\}$ for all $\hat{T}=\hat{T}(R)>0$ small enough). We show that for $\tilde{T}\le T$ small enough, it holds that
	\[u^1,u^2\in B_R^{\tilde{T}}(v_0).\]
	We have that
	\[\|u^i(t,.)-u_0\|_{C^0(M)},\|\nabla u^i(t,.)-\nabla u_0\|_{C^0(M)}\to 0\]
	for $t\to 0$.\footnote{This can be seen as follows: we write $u=u^i$. Since $u\in C^{1,2,\alpha}((0,T)\times M)\subset C^{0,\frac{\alpha}{2}}((0,T);C^2(M))$ (c.f. \cite{BDPhil}) we have in particular that $u,\nabla u\colon (0,T)\to C^0(M)$ are $\frac{\alpha}{2}$-Hölder continuous. (In the case of $\nabla u$ we write $C^0(M)$ as target space shortly for $\Gamma(TM)$ with the $C^0$-norm.) Hence $u,\nabla u\colon (0,T)\to C^0(M)$ are uniformly continuous and can therefore be continuously extended to $u,\nabla u\colon [0,T]\to C^0(M)$. Hence $u(t,.)\to u_0$ in $C^0(M)$ as $t\to0$ and there exists a vector field $V\in\Gamma(TM)$ s.t. $\nabla u(t,.)\to V$ in $C^0(M)$ as $t\to 0$. We show $V=\nabla u_0$. To that end, notice that for every $X\in\Gamma(TM)$ we have \begin{align*}
		\int_M\langle \nabla u(t,.),X\rangle=-\int_Mu(t,.)\textup{div}(X)\xrightarrow{t\to 0} -\int_M u_0\textup{div}(X)=\int_M\langle \nabla u_0,X\rangle,
		\end{align*}
		and
		\[	\int_M\langle \nabla u(t,.),X\rangle\xrightarrow{t\to 0}\int_M\langle V,X\rangle.\]} Moreover we have 
	\[\|v_0(t,.)-u_0\|_{C^0(M)},\|\nabla v_0(t,.)-\nabla u_0\|_{C^0(M)}\to 0\]
	for $t\to 0$. Therefore for ${\tilde{T}}>0$ small enough it holds that
	\[\|u^i-v_0\|_{X_{\tilde{T}}}<R,\]
	and
	\[u^i\in B_R^{\tilde{T}}(v_0)\cap\{u|_{t=0}=u_0\}.\]
	Since $\textup{dim}_\mathbb{K}(\slashed{D}^{\pi\circ u^i_t})=1$ on $[0,T]$ we have that 
	\[\psi_t^i=\psi(u^i_t)h^i_t\]
	for all $t\in [0,\tilde{T}]$ where $\psi(u^i_t)$ is defined as in Lemma \ref{lemma zentrale Spinorabschaetzung} and $h^i_t\in \mathbb{K}$ for all $t\in[0,\tilde{T}]$ and $i=1,2$. Moreover, $h^i_t$ is of unit length since $\|\psi^i_t\|_{L^2(M)}=\|\psi(u^i_t)\|_{L^2(M)}=1$. Since the (real part of the) bundle metric on $\Sigma M$ is invariant under multiplication with elements of $\mathbb{K}$ of unit length, c.f. Lemma \ref{lemma bundle metric inv. unit quat} for the case $\mathbb{K}=\mathbb{H}$, we have that
	\[F_2(u^i,\psi^i)=F_2(u^i,\psi(u^i))\]
	on $[0,\tilde{T}]\times M$. In summary we have shown that $u_1$ and $u_2$ are elements of $B_R^{\tilde{T}}(v_0)$ that solve \eqref{eq z10}. Since the fixed point we constructed in step 1 is unique, we have that $u_1=u_2$ on $[0,\tilde{T}]\times M$ for $\tilde{T}>0$ small enough. Next we define
	\[T_0:=\sup\{t\in[\tilde{T},T]\text{ }|\text{ } u^1=u^2 \text{ on } [0,t]\times M\}\]
	By the definition of $T_0$ and continuity we have $u^1=u^2$ on $[0,T_0]\times M$. We show $T_0=T$. To that end we argue by contradiction and suppose that $T_0<T$. Then $(\hat{u}^i,\hat{\psi}^i)$ defined by $\hat{u}^i(t,x):=u^i(t+T_0,x)$, $(t,x)\in[0,T-T_0]\times M$, $\hat{\psi}^i_t:=\psi^i_{t+T_0}$ are solutions of the heat flow for Dirac harmonic maps with $T$ replaced by $T-T_0$, $u_0$ replaced by $\hat{u}_0$, where $\hat{u}_0:=u^1(T_0,.)=u^2(T_0,.)$, and $\psi_0$ replaced by $\hat{\psi}_0:=\psi^1_{T_0}=\psi^2_{T_0}$.\footnote{Since $\psi^1_{T_0},\psi^2_{T_0}\in \textup{ker}(\slashed{D}^{u^1_{T_0}})$ and $\textup{dim}_{\mathbb{K}}\textup{ker}(\slashed{D}^{u^1_{T_0}})=1$, we can assume w.l.o.g. that $\psi^1_{T_0}=\psi^2_{T_0}$. Otherwise we replace $\psi^2$ by $\psi^2h$, where $h\in\mathbb{K}$ has unit length with $\psi^1_{T_0}=\psi^2_{T_0}h$.} Using the preceding argument we get that there exists some $\tilde{T}>T_0$ s.t. $u^1=u^2$ on $[T_0,\tilde{T}]\times M$. This contradicts the definition of $T_0$. Therefore $T_0=T$.	
\end{proof}

\newpage
\bibliographystyle{abbrv}
\bibliography{references}

\end{document}